\newif \ifIMA
\ifIMA
\documentclass[numbers,webpdf,imaiai]{ima-authoring-template}%

\usepackage[mathlines, right]{lineno}
\linenumbers

\theoremstyle{thmstyletwo}%
\newtheorem{theorem}{Theorem}[section]% 
\newtheorem{proposition}[theorem]{Proposition}%
\newtheorem{lemma}[theorem]{Lemma}%

\newtheorem{remark}[theorem]{Remark}%

\newtheorem{definition}[theorem]{Definition}

\numberwithin{equation}{section}

\else
\documentclass{amsart}
\usepackage{geometry,graphicx,amssymb,amsmath,amsbsy,amsfonts,mathrsfs,amscd,tcolorbox, xcolor,caption,float,cancel}
\usepackage[textsize=tiny,color=blue!50!white]{todonotes}
\usepackage[normalem]{ulem} % for strikethrough (\sout) in collaborative writing

\numberwithin{equation}{section}

\allowdisplaybreaks[3]

\newtheorem{theorem}{Theorem}[section]
\newtheorem{lemma}[theorem]{Lemma}

\newtheorem{proposition}[theorem]{Proposition}
\theoremstyle{definition}

\newtheorem{remark}[theorem]{Remark}

\usepackage[breaklinks,bookmarks=false]{hyperref}

\hypersetup{colorlinks, linkcolor=blue, citecolor=blue,
urlcolor=blue, plainpages=false, pdfwindowui=false,
pdfstartview={FitH}, pdftitle={},}

\fi

\usepackage{bm,amsmath}

\usepackage{mydef}

\ifIMA

\DeclareMathAlphabet\mathbfcal{OMS}{cmsy}{b}{n}

\renewcommand{\Rt}{\mathbfcal{R\hspace{-0.1em}T}\hspace{-0.25em}}
\renewcommand{\Ne}{\mathbfcal{N}}

\fi

\newcommand\cf{cf.}
\newcommand\eg{e.g.}
\newcommand\eal{{\em et al.}}

\newcommand\crl{\mathrm{curl}}

\newcommand\nn{\nonumber}

\newcommand\ver{v}
\newcommand\bxih{\bm{\xi}\lT}
\newcommand\frh{\bm{\sigma}\lT} % flux reconstruction
\newcommand\psia{\psi^\ver}
\newcommand\oma{\omega_\ver}
\newcommand\Tha{\mathcal{T}_\ver}
\newcommand\RTproj[1]{\bm{I}^{\bm{\mathcal{R\hspace{-0.1em}T}}}_{#1}}
\newcommand\Hdva{\mbH(\dive,\oma)}
\newcommand\refv[2]{\stackrel{#1}{#2}}

\begin{document}

\ifIMA

\DOI{DOI HERE}
\copyrightyear{2024}
\vol{00}
\pubyear{2024}
\access{Advance Access Publication Date: Day Month Year}
\appnotes{Paper}
\copyrightstatement{Published by Oxford University Press on behalf of the Institute of Mathematics and its Applications. All rights reserved.}
\firstpage{1}

\fi

\def\AEr{Alexandre Ern}
\def\AEa{CERMICS, ENPC, Institut Polytechnique de Paris, 77455 Marne-la-Vall\'ee, France \& Inria Paris, 48 rue Barrault, 75647 Paris, France}

\def\JG{Johnny Guzm\'an}
\def\JGa{Division of Applied Mathematics,
Brown University,
Box F,
182 George Street,
Providence, RI 02912, USA}

\def\PP{Pratyush Potu}
\def\PPa{Division of Applied Mathematics,
Brown University,
Box F,
182 George Street,
Providence, RI 02912, USA}

\def\MV{Martin Vohral\'ik\ifIMA*\fi}
\def\MVa{Inria Paris, 48 rue Barrault, 75647 Paris, France \& CERMICS, ENPC, Institut Polytechnique de Paris, 77455 Marne-la-Vall\'ee, France}

\title[Discrete Poincar\'e inequalities]{Discrete Poincar\'e inequalities: a review on proofs, equivalent formulations, and behavior of constants}

\ifIMA 

\author{\AEr
\address{\AEa}}

\author{\JG
\address{\JGa}}

\author{\PP
\address{\PPa}}

\author{\MV
\address{\MVa}}

\else

\author{\AEr}
\address{\AEa}
\email{\href{mailto:alexandre.ern@enpc.fr}{alexandre.ern@enpc.fr}}

\author{\JG}
\address{\JGa}
\email{\href{mailto:johnny_guzman@brown.edu}{johnny\_guzman@brown.edu}}

\author{\PP}
\address{\PPa}
\email{\href{mailto:pratyush_potu@brown.edu}{pratyush\_potu@brown.edu}}

\author{\MV}
\address{\MVa}
\email{\href{mailto:martin.vohralik@inria.fr}{martin.vohralik@inria.fr}}

\fi

%\thanks{}

\ifIMA

\authormark{Alexandre Ern et al.}

\corresp[*]{Corresponding author: \href{email:martin.vohralik@inria.fr}{martin.vohralik@inria.fr}}

\received{Date}{0}{Year}
\revised{Date}{0}{Year}
\accepted{Date}{0}{Year}

\else
\fi

\def\abs{We investigate discrete Poincar\'e inequalities on piecewise polynomial subspaces of the Sobolev spaces $\Hcrl$ and $\Hdiv$ in three space dimensions. We characterize the dependence of the constants on the continuous-level constants, the shape regularity and cardinality of the underlying tetrahedral mesh, and the polynomial degree. One important focus is on meshes being local patches (stars) of tetrahedra from a larger tetrahedral mesh. We also review various equivalent results to the discrete Poincar\'e inequalities, namely stability of discrete constrained minimization problems, discrete inf-sup conditions, bounds on operator norms of piecewise polynomial vector potential operators (Poincar\'e maps), and existence of graph-stable commuting projections.}

\ifIMA

\abstract{\abs}

\else

\begin{abstract}
\abs
\end{abstract}
\fi


\maketitle

% \tableofcontents

%%%%%%%%%%%%%%%%%%%%%%%%%%%%%%%%%%%%%
%%%%%%%%%%%%%%%%%%%%%%%%%%%%%%%%%%%%%

\section{Introduction}

Let $\omega$ be a three-dimensional, open, bounded, connected, Lipschitz polyhedral domain with diameter $h_\omega$. The $L^2$-inner product in $\omega$ is denoted as $\bl\cdot,\cdot\br_\omega$ and the corresponding norm as $\|{\cdot}\|_{L^2(\omega)}$ or $\|{\cdot}\|_{\bL^2(\omega)}$ (notation is set in details in Section~\ref{sec_not} below). Let $\Tom$ be a tetrahedral mesh of $\omega$. Our main motivation is the case where $\Tom$ is some local collection (patch, star) of tetrahedra from a mesh of some larger fixed three-dimensional domain, say $\Omega$.
A large part of our developments applies in any space dimension in the language of differential forms, but we choose the three-dimensional setting to hopefully address a broad readership.

\subsection{Poincar\'e and discrete Poincar\'e inequalities on $\Hgrd = H^1(\omega)$}

The Poincar\'e inequality 
\begin{equation}
\|u\|_{L^2(\omega)} \le \CP^0 h_\omega \| \grad \, u \|_{\bL^2(\omega)}  \qquad \forall u \in \Hgrd \text{ such that } \bl u, 1\br_\omega = 0  \label{c_onto0_I}
\end{equation}
is well known and omnipresent in the analysis of partial differential equations. Crucially, $\CP^0$ is a generic constant that only depends on the shape of $\omega$. For convex $\omega$, in particular, $\CP^0 \leq 1 / \pi$ following Payne and Weinberger~\cite{Pay_Wei_Poin_conv_60} and Bebendorf~\cite{Beben_Poin_conv_03}.
The discrete version of~\eqref{c_onto0_I} for $\Hgrd$-conforming piecewise polynomials of degree $(p+1)$, $p \geq 0$, on the tetrahedral mesh $\Tom$ of $\omega$ writes
\begin{equation}
\|u\lT\|_{L^2(\omega)} \le \CPd{0} h_\omega \| \grad \, u\lT \|_{\bL^2(\omega)}  \qquad \forall u\lT \in \pol_{p+1}(\Tom) \cap \Hgrd \text{ such that } \bl u\lT, 1\br_\omega = 0.
\label{c_onto0_I_disc}
\end{equation}
As the functions considered in~\eqref{c_onto0_I_disc} form a subspace of the functions considered in~\eqref{c_onto0_I}, \eqref{c_onto0_I_disc} trivially holds with $\CPd{0} \le \CP^0$. Moreover, similar results hold with (homogeneous) boundary condition on the boundary $\partial \omega$. 

\subsection{Poincar\'e and discrete Poincar\'e inequalities on \texorpdfstring{$\Hcrl$}{H(curl)} and \texorpdfstring{$\Hdiv$}{H(div)}}

The Poincar\'e inequalities
\begin{subequations} \label{c_onto_I}
\begin{alignat}{2}
\|\bu\|_{\bL^2(\omega)} &\le \CP^1 h_\omega \|\curl \, \bu\|_{\bL^2(\omega)} \qquad && \forall  \bu \in \Hcrl \text{ such that } \bl \bu , \bv\br_\omega = 0 \nn \\
& && \qquad \forall \bv \in \Hcrl  \text{ with } \curl \, \bv=\bzero, \label{c_onto1_I} \\
\|\bu\|_{\bL^2(\omega)} &\le \CP^2 h_\omega \|\dive \, \bu\|_{L^2(\omega)},   \qquad && \forall  \bu \in \Hdiv  \text{ such that } \bl \bu, \bv\br_\omega = 0 \nn \\
& && \qquad \forall \bv \in \Hdiv \text{ with } \dive \, \bv=0  \label{c_onto2_I}
\end{alignat}
\end{subequations}
arise when the differential operators employ the curl or the divergence in place of the gradient. They are also well known, though a bit more complicated to establish.
When $\omega$ is simply connected, the orthogonality in~\eqref{c_onto1_I}
means that $\bu$ is orthogonal to gradients of $\Hgrd$ functions and thus belongs to $\Hdiv$, is divergence-free, and has zero normal component on the boundary $\partial\omega$. Thus, \eqref{c_onto1_I} is the so-called Poincar\'e--Friedrichs--Weber 
inequality, see Fernandes and Gilardi~\cite[Proposition 7.4]{Fer_Gil_Maxw_BC_97} 
or Chaumont-Frelet et al.~\cite[Theorem A.1]{Chaum_Ern_Voh_Maxw_22}. Similarly, when
$\partial\omega$ is connected, the orthogonality in~\eqref{c_onto2_I}
means that $\bu$ is orthogonal to curls of $\Hcrl$ functions and thus belongs to $\Hcrl$, 
is curl-free, and has zero tangential component on the boundary
$\partial\omega$.

In this paper, we are interested in the following {\em discrete Poincar\'e inequalities} for $\Hcrl$- and $\Hdiv$-conforming piecewise polynomials in the N\'ed\'elec and Raviart--Thomas finite element spaces of order $p$, $p\ge0$, on the tetrahedral mesh $\Tom$ of $\omega$:
\begin{subequations} \label{c_onto_I_disc}
\begin{alignat}{2}
\|\bu\lT\|_{\bL^2(\omega)} &\le \CPd{1} h_\omega \|\curl \, \bu\lT\|_{\bL^2(\omega)} \qquad && \forall  \bu\lT \in \Ne_p(\Tom) \cap \Hcrl \text{ such that } \bl \bu\lT , \bv\lT\br_\omega = 0 \nn \\
& && \qquad \forall \bv\lT \in \Ne_p(\Tom) \cap \Hcrl \text{ with } \curl \, \bv\lT=\bzero, \label{c_onto1_I_disc} \\
\|\bu\lT\|_{\bL^2(\omega)} &\le \CPd{2} h_\omega \|\dive \, \bu\lT\|_{L^2(\omega)},   \qquad && \forall  \bu\lT \in \Rt_p(\Tom) \cap \Hdiv \text{ such that } \bl \bu\lT, \bv\lT\br_\omega = 0 \nn \\
& && \qquad \forall \bv\lT \in \Rt_p(\Tom) \cap \Hdiv \text{ with } \dive \, \bv\lT=0. \label{c_onto2_I_disc}
\end{alignat}
\end{subequations}
We will also consider the counterparts with homogeneous boundary condition on the boundary $\partial \omega$. 
Here, unfortunately, the inequalities~\eqref{c_onto_I_disc} do not follow from~\eqref{c_onto_I} and $\CPd{1}$, $\CPd{2}$ cannot be trivially bounded by $\CP^1$, $\CP^2$. Indeed, the kernel of the $\curl$ operator restricted to $\Ne_p(\Tom) \cap \Hcrl$ is different from the kernel on $\Hcrl$, and similarly, the kernel of the $\dive$ operator restricted to $\Rt_p(\Tom) \cap \Hdiv$ is different from the kernel on $\Hdiv$ (the former being nontrivial and the latter being infinite-dimensional in both cases). In contrast, the kernel of the gradient operator is trivial (composed of constant functions) and is the same on $\pol_{p+1}(\Tom) \cap \Hgrd$ and $\Hgrd$, which leads to the trivial passage from~\eqref{c_onto0_I} to~\eqref{c_onto0_I_disc}.

\begin{remark}[Orthogonality constraint]
When $\omega$ is simply connected, a vector-valued piecewise polynomial $\bv\lT \in \Ne_p(\Tom) \cap \Hcrl$ with $\curl \, \bv\lT=\bzero$ is the gradient of a scalar-valued piecewise polynomial $v\lT \in \pol_{p+1}(\Tom) \cap \Hgrd$; the orthogonality constraint in~\eqref{c_onto1_I_disc} is often stated in the literature using gradients. We prefer the writing~\eqref{c_onto_I_disc}, since the orthogonality constraints in~\eqref{c_onto_I_disc} are valid for a general domain topology. 
\end{remark}

\subsection{Focus of the paper}

In the literature, one often finds assertions that the discrete Poincar\'e inequalities~\eqref{c_onto_I_disc} are  ``known''. 
The purpose of this paper is to recall several equivalent reformulations of~\eqref{c_onto_I_disc}, discuss the available references, formulate some possible proofs of~\eqref{c_onto_I_disc} in an abstract way with generic assumptions, and to establish some new results on~\eqref{c_onto_I_disc}. Our main focus is on the {\em characterization} of the {\em behavior} of the constants $\CPd{1}$, $\CPd{2}$ with respect to the {\em constants} $\CP^1$, $\CP^2$, the {\em shape-regularity parameter} of the mesh $\Tom$, the {\em number of elements} in $\Tom$, and the {\em polynomial degree} $p$. 
The motivation for writing explicitly the scaling with $h_\omega$ in~\eqref{c_onto0_I_disc} and~\eqref{c_onto_I_disc} is twofold: (i) it is important when $\omega$ corresponds to a local collection of tetrahedra from a larger mesh; (ii) it makes the constants $\CP^l$ and $\CPd{l}$, $l\in\{0,1,2\}$, dimensionless.

\subsection{Available results}

As discussed in Sections~\ref{sec_equivs} and~\ref{sec_route_2} below in more detail, the discrete Poincar\'e inequalities~\eqref{c_onto_I_disc} are {\em equivalent} to: (i) stability of discrete constrained minimization problems; (ii) discrete inf-sup conditions; (iii) bounds on operator norms of piecewise polynomial vector potential operators (that is, piecewise polynomial right-inverses for the curl and divergence operators, also called Poincar\'e maps); and (iv) existence of graph-stable commuting projections. There are also links to lower bounds on eigenvalues of vector Laplacians. Numerous results are available in the literature in one of these settings. 

The discrete Poincar\'e inequality~\eqref{c_onto1_I_disc}, with a generic constant $C$ in place of $\CPd{1} h_\omega$, is presented in Girault and Raviart~\cite[Chapter~3, Proposition~5.1]{Gir_Rav_NS_86} and Monk and Demkowicz~\cite[Corollary~4.2]{Monk_Demk_disc_inf_sup_Maxw_00} in three space dimensions and in Arnold~\eal\ \cite[Theorem~5.11]{arnold2006finite} and Arnold~\eal\ \cite[Theorem~3.6]{arnold2010finite} more abstractly in the finite element exterior calculus setting, covering both bounds in~\eqref{c_onto_I_disc}. The discrete Poincar\'e inequality in the precise form~\eqref{c_onto1_I_disc} is established in Ern and Guermond~\cite[Theorem~44.6]{EG_volII}, with $\CPd{1}$ at worst depending on the continuous-level constant $\CP^1$ from~\eqref{c_onto1_I}, the shape-regularity parameter of $\Tom$, and the polynomial degree $p$, see also~\cite[Remark~44.7]{EG_volII} for further bibliographical resources.

Discrete inf-sup conditions are extensively discussed in the mixed finite element literature.
For instance, \eqref{c_onto2_I_disc} as a discrete inf-sup condition is established, with a generic constant $C$ in place of $\CPd{2} h_\omega$, in Raviart and Thomas~\cite[Theorem~4]{Ra_Tho_MFE_77}, see also Fortin~\cite{Fort_MFEs_77}, Boffi~\eal\ \cite[Theorem~4.2.3 and Propositions~5.4.3 and~7.1.1]{BBF13}, or Gatica~\cite[Lemmas~2.6 and 4.4.]{Gat_MFEs_14}. The form leading precisely to~\eqref{c_onto2_I_disc} can be found in~\cite[Remark~51.12]{EG_volII}, with $\CPd{2}$ at worst depending on the continuous-level constant $\CP^2$ from~\eqref{c_onto2_I}, the shape-regularity parameter of $\Tom$, and the polynomial degree $p$.

Considering the operator norm of a piecewise polynomial vector potential operator, Demkowicz and Babu\v{s}ka~\cite[Theorem~1]{Demk_Bab_hp_Ned_2D_03}, Gopalakrishnan and Demkowicz~\cite[Theorems~4.1, 5.1, and~6.1]{Gop_Demk_q_opt_MFEs_04}, and Demkowicz and Buffa~\cite[Lemmas~6 and~8]{Demk_Buf_q_opt_proj_int_05} establish~\eqref{c_onto_I_disc} with a generic constant $C$ independent of the polynomial degree $p$ ({\em $p$-robustness}) on a single triangle or tetrahedron. Similar results hold on a cube and more generally on starlike domains with respect to a ball, see Costabel~\eal\ \cite{Cost_Daug_Demk_ext_08} and Costabel and McIntosh~\cite{Cost_McInt_Bog_Poinc_10}. Unfortunately, none of these results addresses piecewise polynomials with respect to a mesh $\Tom$. This issue is discussed in Boffi~\eal\ \cite[Lemma~2.5]{Boff_Cost_Daug_Demk_Hipt_inf_sup_Maxw_p_11} for the $p$-version finite element method on a fixed mesh. 

Piecewise polynomials on patches of tetrahedra sharing a given subsimplex (vertex, edge, or face) seem to have been addressed only more recently. Corresponding proofs employ the above-discussed results for polynomials on one element together with polynomial extension operators from the boundary of a tetrahedron (Demkowicz~\eal\ \cite{Demk_Gop_Sch_ext_II_09, Demk_Gop_Sch_ext_III_12} for, respectively, the tangential or normal trace lifting in the $\Hcrl$ or $\Hdiv$ context; \cf\ also the recent work of Falk and Winther~\cite{Falk_Wint_pol_exts_23} for a $d$-simplex). Following some early contributions as Gopalakrishnan \eal\ \cite[Lemma~3.1 and Appendix]{Gopalakrishnan2004}, Braess~\eal\ \cite{Brae_Pill_Sch_p_rob_09} address vertex stars in 2D and Ern and Vohral\'\i k~\cite[Corollaries~3.3 and~3.8]{Ern_Voh_p_rob_3D_20} consider vertex stars in 3D in the $\Hdiv$ context, whereas the $\Hcrl$ context is developed in Chaumont-Frelet et al.~\cite[Theorem~3.1]{Chaum_Ern_Voh_Maxw_22} and Chaumont-Frelet and Vohral\'\i k~\cite[Theorem~3.3 and Corollary~4.3]{Chaum_Voh_p_rob_3D_H_curl_24} (respectively edge and vertex stars in 3D). As we shall see, these results imply~\eqref{c_onto_I_disc} with $\CPd{1}$, $\CPd{2}$ being $p$-robust, but possibly depending on the number of elements in the mesh $\Tom$. Finally, simultaneous independence of the number of elements in the mesh $\Tom$ and of the polynomial degree $p$ follows from the recent result of Demkowicz and Vohral{\'{\i}}k~\cite{Demk_Voh_loc_glob_H_div_25}.

\subsection{Main results and organization of the paper}

We introduce the setting in Section~\ref{sec_not} together with a unified notation to formulate the Poincar\'e inequalities without the need to distinguish between grad, curl, and div operators.
In Section~\ref{sec_equivs} we recall that discrete Poincar\'e inequalities are equivalent with stability of discrete constrained minimization problems, discrete inf-sup conditions, and bounds on operator norms of piecewise polynomial vector potential operators. Section~\ref{sec_Poinc_cont} then wraps up known results on the continuous Poincar\'e inequalities~\eqref{c_onto0_I} and~\eqref{c_onto_I} and their variants with boundary conditions on $\partial \omega$. Turning next to the discrete Poincar\'e inequalities in Section~\ref{sec_Poinc_disc},
our main result is Theorem~\ref{thm_disc_Poinc}, establishing~\eqref{c_onto_I_disc} and its variants with boundary conditions on $\partial \omega$. In particular, we thoroughly discuss the dependencies of $\CPd{1}$, $\CPd{2}$ on the constants $\CP^1$, $\CP^2$, the shape-regularity parameter of $\Tom$, the number of elements in $\Tom$, and the polynomial degree $p$. Three different proofs, leading to various dependencies, are presented in Section~\ref{sec_dis_Poinc_proofs}, relying either on available results from the literature (invoking equivalences between discrete and continuous minimizers or stable commuting projections) or on a self-standing proof invoking piecewise Piola transformations. In Section~\ref{sec_dis_Poinc_proofs}, we also recall the equivalence of discrete Poincar\'e inequalities with the existence of graph-stable commuting projections.
% Finally, in order to make the manuscript self-contained, Appendix~\ref{sec_Poinc_cont_compl} showcases some possible proofs of, and collects some additional observations on, the Poincar\'e inequalities~\eqref{c_onto_I} at the continuous level.

\section{Setting and compact notation} \label{sec_not}

Let $\omega$ be a three-dimensional, open, bounded, connected, Lipschitz polyhedral domain with boundary
$\partial\omega$ and unit outward normal $\bn_\omega$.
Let $h_\omega$ denote the diameter of $\omega$. 
We use boldface font for vector-valued quantities, vector-valued fields, and
functional spaces composed of such fields.
For simplicity, the inner product in $L^2(\omega)$ and $\bL^2(\omega)$ is
abbreviated as $\bl\cdot,\cdot\br_\omega$, whereas the norms are written as $\|{\cdot}\|_{L^2(\omega)}$, $\|{\cdot}\|_{\bL^2(\omega)}$. 

\subsection{Sobolev spaces} \label{sec_Sob}

Let $\Hgrd \eq H^1(\omega)$ be the standard Sobolev space of scalar-valued functions from $L^2(\omega)$ with weak gradient in $\bL^2(\omega)$, $\Hcrl$ the Sobolev space of vector-valued functions from $\bL^2(\omega)$ with weak curl in $\bL^2(\omega)$, and $\Hdiv$ the Sobolev space of vector-valued functions from $\bL^2(\omega)$ with weak divergence in $L^2(\omega)$, \cf, \eg, \cite[Sections~2.2--2.3]{Gir_Rav_NS_86} and~\cite[Section~4.3]{ErnGuermondbook}.
These spaces are Hilbert spaces when equipped with the graph norms 
\begin{subequations} \begin{align}
    \|u\|_{\Hgrd}^2 & \eq \|u\|_{L^2(\omega)}^2 + h_\omega^2 \|\grad \, u\|_{\bL^2(\omega)}^2, \\
    \|\bu\|_{\Hcrl}^2 & \eq \|\bu\|_{\bL^2(\omega)}^2 + h_\omega^2 \|\curl \, \bu\|_{\bL^2(\omega)}^2, \\ 
    \|\bu\|_{\Hdiv}^2 & \eq \|\bu\|_{\bL^2(\omega)}^2 + h_\omega^2 \|\dive \, \bu\|_{L^2(\omega)}^2.
\end{align} \end{subequations}
The length scale $h_\omega$ is used for dimensional consistency (in terms of physical units) and corresponds to the scaling in the Poincar\'e inequalities~\eqref{c_onto_I} and~\eqref{c_onto_I_disc}.
We denote by $\Hgrdz \eq H^1_0(\omega)$, $\Hcrlz$, and $\Hdivz$ the subspaces with homogeneous boundary conditions imposed along $\partial \omega$ with the usual trace maps associated with the trace, the trace of the tangential component, and the trace of the normal component on $\partial\omega$. Specifically, for a smooth function or field, the trace maps are $\gamma^0_{\partial\omega}(u)=u|_{\partial\omega}$, 
$\gamma^{1}_{\partial\omega}(\bu)=\bu|_{\partial\omega}{\times}\bn_\omega$, and
$\gamma^2_{\partial\omega}(\bu)=\bu|_{\partial\omega}\SCAL\bn_\omega$.

\subsection{Mesh and piecewise polynomial spaces} \label{sec_pw_pol}

Let $\Tom$ be a triangulation of $\omega$ consisting of a finite number of tetrahedra. The shape-regularity parameter of $\Tom$ is defined as 
\begin{equation} \label{eq_rho}
    \rho_{\Tom} \eq \max_{\tau \in \Tom} h_\tau / \iota_\tau,
\end{equation}
where $h_\tau$ is the diameter of $\tau$ and $\iota_\tau$ the diameter of the largest ball inscribed in $\tau$. We also denote by $|\Tom|$ the cardinality of $\Tom$, i.e., the number of elements in $\Tom$. Let $p\ge0$ be a fixed polynomial degree. For a tetrahedron $\tau \in \Tom$, let $\pol_p(\tau)$ denote the space of polynomials of total degree at most $p$ on $\tau$, $\pol_p(\tau;\R^3)$ its vector-valued counterpart,
\begin{equation}\label{eq_Ned}
    \Ne_p(\tau)\eq\{\bu(\bx) + \bx\times \bv(\bx) : \bu,\bv\in \pol_p(\tau;\R^3)\}
\end{equation}
the $p$-th order N\'ed\'elec space~\cite{nedelec1980mixed}, and 
\begin{equation}\label{eq_RT}
    \Rt_p(\tau)\eq\{\bu(\bx)+v(\bx)\bx: \bu\in \pol_p(\tau;\R^3), v\in\pol_p(\tau)\}
\end{equation}
the $p$-th order Raviart--Thomas space~\cite{Ra_Tho_MFE_77}. We denote the broken spaces (that is, discontinuous piecewise polynomial, without any continuity requirement across the mesh interfaces) as
\begin{subequations} \label{eq:broken_spaces} \begin{align}
    \pol_{p+1}(\Tom) &\eq \{ u\lT\in L^2(\omega): u\lT|_\tau \in \pol_{p+1}(\tau) \forall \tau \in \Tom\}, \\
    \Ne_p(\Tom) &\eq  \{ \bu\lT\in \bL^2(\omega): \bu\lT|_\tau \in \Ne_p(\tau) \forall \tau \in \Tom\}, \\
    \Rt_p(\Tom) &\eq \{ \bu\lT\in \bL^2(\omega): \bu\lT|_\tau \in \Rt_p(\tau) \forall \tau \in \Tom\}. 
\end{align} \end{subequations}
The usual subspaces with continuous trace, tangential trace, and normal trace are $\pol_{p+1}(\Tom) \cap \Hgrd$, $\Ne_p(\Tom) \cap \Hcrl$, and
$\Rt_p(\Tom) \cap \Hdiv$. We proceed similarly for the ho\-mo\-ge\-ne\-ous-trace subspaces.
Here and in what follows, the subscript ${}\lT$ generically refers to functions and fields that sit in the above finite-dimensional spaces. 

\subsection{Compact notation} \label{sec_comp_not}

We introduce here a compact notation that allows us to present the subsequent developments in a unified setting.

At the continuous level, we denote
\begin{subequations} \label{eq:graph_spaces} \begin{alignat}{2}
V^0(\omega)&\eq \Hgrd, &\qquad \mV^0(\omega)&\eq \Hgrdz,\\ 
\bV^1(\omega)&\eq \Hcrl, &\qquad \mbV^1(\omega)&\eq \Hcrlz,\\ 
\bV^2(\omega)&\eq \Hdiv, &\qquad \mbV^2(\omega)&\eq \Hdivz,\\ 
V^3(\omega)&\eq L^2(\omega), &\qquad \mV^3(\omega)&\eq \mL^2(\omega)\eq\{ u\in L^2(\omega) \;:\; \bl u,1\br_\omega=0\}. 
\end{alignat} \end{subequations}
With this notation, we have the well-known de Rham sequences
\begin{subequations} \label{complexes} \begin{equation}\label{complex}
\mathbb{R}
\stackrel{\subset}{\xrightarrow{\hspace*{0.5cm}}}\
 V^0(\omega)
\stackrel{\grad}{\xrightarrow{\hspace*{0.5cm}}}\
 \bV^1(\omega)
\stackrel{\curl}{\xrightarrow{\hspace*{0.5cm}}}\
 \bV^2(\omega)
\stackrel{\dive}{\xrightarrow{\hspace*{0.5cm}}}\
V^3(\omega)
\stackrel{}{\xrightarrow{\hspace*{0.5cm}}}
0,
\end{equation}
\begin{equation}\label{complex_m}
0
\stackrel{\subset}{\xrightarrow{\hspace*{0.5cm}}}\
 \mV^0(\omega)
\stackrel{\grad}{\xrightarrow{\hspace*{0.5cm}}}\
 \mbV^1(\omega)
\stackrel{\curl}{\xrightarrow{\hspace*{0.5cm}}}\
 \mbV^2(\omega)
\stackrel{\dive}{\xrightarrow{\hspace*{0.5cm}}}\
\mV^3(\omega)
\stackrel{\int_\omega}{\xrightarrow{\hspace*{0.5cm}}}
0,
\end{equation} \end{subequations}

Similarly, at the discrete level, we denote
\begin{subequations} \label{eq:local_spaces} \begin{alignat}{2}
V_p^0(\Tom) &\eq  \pol_{p+1}(\Tom) \cap \Hgrd, &\qquad \mV_p^0(\Tom) &\eq \pol_{p+1}(\Tom) \cap \Hgrdz, \\
\bV_p^1(\Tom) &\eq  \Ne_p(\Tom) \cap \Hcrl,  &\qquad \mbV_p^1(\Tom) &\eq \Ne_p(\Tom) \cap \Hcrlz,\\
\bV_p^2(\Tom) &\eq \Rt_p(\Tom) \cap \Hdiv,  &\qquad \mbV_p^2(\Tom) &\eq \Rt_p(\Tom) \cap \Hdivz,\\
V_p^3(\Tom) &\eq \pol_{p}(\Tom),  &\qquad \mV_p^3(\Tom) &\eq \pol_{p}(\Tom) \cap \mL^2(\omega).
\end{alignat} \end{subequations}
As in~\eqref{complexes}, the discrete spaces are related by the following two discrete de Rham sequences:
\begin{subequations} \begin{equation} \label{complex_h}
\mathbb{R}
\stackrel{\subset}{\xrightarrow{\hspace*{0.5cm}}}\
 V_p^0(\Tom)
\stackrel{\grad}{\xrightarrow{\hspace*{0.5cm}}}\
  \bV_p^1(\Tom)
\stackrel{\curl}{\xrightarrow{\hspace*{0.5cm}}}\
  \bV_p^2(\Tom)
\stackrel{\dive}{\xrightarrow{\hspace*{0.5cm}}}\
 V_p^3(\Tom)
\stackrel{}{\xrightarrow{\hspace*{0.5cm}}} 0,
\end{equation}
\begin{equation} \label{complex_hm}
0 
\stackrel{\subset}{\xrightarrow{\hspace*{0.5cm}}}\
 \mV_p^0(\Tom)
\stackrel{\grad}{\xrightarrow{\hspace*{0.5cm}}}\
  \mbV_p^1 (\Tom)
\stackrel{\curl}{\xrightarrow{\hspace*{0.5cm}}}\
  \mbV_p^2(\Tom)
\stackrel{\dive}{\xrightarrow{\hspace*{0.5cm}}}\
 \mV_p^3(\Tom)
\stackrel{\int_\omega}{\xrightarrow{\hspace*{0.5cm}}}
0.
\end{equation}\end{subequations}

We henceforth use the generic notation $V^l(\omega)$, $\mV^l(\omega)$ for the continuous 
spaces defined in~\eqref{eq:graph_spaces} and $V_p^l(\Tom)$, 
$\mV_p^l(\Tom)$ for their discrete subspaces defined in~\eqref{eq:local_spaces} with $l\in\{0{:}3\}$. Moreover, we define
\begin{equation} \label{eq_dl}
    d^0\eq\grad, \quad d^1\eq\curl, \quad d^2\eq\dive.
\end{equation} 
We also use $\| \SCAL \|_{L^2(\omega)}$ to generically refer to the $L^2(\omega)$-norm or $\bL^2(\omega)$-norm of functions or fields depending on the context.

\subsection{Compact writing of Poincar\'e inequalities}

With the above notation, the continuous Poincar\'e inequalities~\eqref{c_onto0_I} and~\eqref{c_onto_I} are rewritten as follows:
\begin{alignat}{2}
\|u\|_{L^2(\omega)} &\le \CP^l h_\omega \| d^l u \|_{L^2(\omega)} \qquad && \forall u \in V^l(\omega) \text{ such that } \bl u , v\br_\omega = 0 \nn \\
& && \qquad \forall v \in V^l(\omega) \text{ with } d^l v = 0 \qquad \forall l\in\{0{:}2\} \label{c_onto_abstr} 
\end{alignat}
and the discrete Poincar\'e inequalities~\eqref{c_onto0_I_disc} and~\eqref{c_onto_I_disc} are rewritten as follows:
\begin{alignat}{2}
\|u\lT\|_{L^2(\omega)} &\le \CPd{l} h_\omega \| d^l u\lT \|_{L^2(\omega)} \qquad && \forall u\lT \in V_p^l(\Tom) \text{ such that } \bl u\lT , v\lT\br_\omega = 0 \nn \\
& && \qquad \forall v\lT \in V_p^l(\Tom) \text{ with } d^l v\lT = 0 \qquad \forall l\in\{0{:}2\}. \label{c_onto_disc_abstr}
\end{alignat}
Similar statements in the case of prescribed boundary conditions can be found in Proposition~\ref{prop_Poinc} and Theorem~\ref{thm_disc_Poinc}.

% This form in particular applies in any space dimension in the language of differential forms.

\section{Equivalent statements for discrete Poincar\'e inequalities}\label{sec_equivs}

In this section, we recall that the discrete Poincar\'e inequalities~\eqref{c_onto_I_disc}, i.e., \eqref{c_onto_disc_abstr} for $l\in\{1{:}2\}$, are equivalent to: (i) stability of discrete constrained minimization problems; (ii) discrete inf-sup conditions; and (iii) bounds on operator norms of piecewise polynomial vector potential operators. All these equivalences are known from the literature, but possibly not that well known, and definitely seldom presented together. We find it instructive to briefly recall them, including proofs. Similar equivalences hold when homogeneous boundary conditions are imposed on the boundary $\partial \omega$ and are not detailed for brevity. These equivalences only consider finite-dimensional spaces and are rather easy to expose. A further equivalence with the existence of graph-stable commuting projections includes the infinite-dimensional spaces $V^l(\omega)$ and requires a bit more setup; we postpone it to Lemma~\ref{lem_eq_stab_proj_dP} below. 
%Let for simplicity the topology of $\omega$ be trivial.

We proceed with the compact notation of Section~\ref{sec_comp_not}. %, which gives a form that actually applies in any space dimension in the language of differential forms. 
For the reader's convenience, we also state the results explicitly for the practically relevant cases of $\Hcrl$ and $\Hdiv$. In particular, the discrete spaces $d^l (V_p^l(\Tom))$ here take the form 
\begin{subequations}
\begin{align}
d^1 (V_p^1(\Tom)) & = \curl (\Ne_p(\Tom) \cap \Hcrl) \subset \{\bv\lT \in \Rt_p(\Tom) \cap \Hdiv \text{ with } \dive \, \bv\lT = 0\}, \label{eq_curl_map}\\
d^2 (V_p^2(\Tom)) & = \dive (\Rt_p(\Tom) \cap \Hdiv) = \pol_p(\Tom). \label{eq_div_map}
\end{align}
\end{subequations}
When the boundary of $\omega$ is connected, we have more precisely
$\curl (\Ne_p(\Tom) \cap \Hcrl) = \{\bv\lT \in \Rt_p(\Tom) \cap \Hdiv \text{ with } \dive \, \bv\lT = 0\}$.

\subsection{Equivalence with stability of discrete constrained minimization problems}

Let $l\in\{1{:}2\}$, $r\lT \in d^l (V_p^l(\Tom))$, and consider the {\em constrained quadratic minimization problem} 
\begin{equation} \label{eq_min}
u\lT^* \eq \Argmin_{\substack{v\lT \in V_p^l(\Tom) \\ d^l v\lT = r\lT}} \|v\lT\|_{L^2(\omega)}^2, 
\end{equation}
Since the minimization set is closed, convex, and nonempty as we suppose $r\lT \in d^l (V_p^l(\Tom))$ and the minimized functional is continuous and strongly convex, the above problem has a unique solution.

\begin{lemma}[Equivalence of~\eqref{c_onto_disc_abstr} with stability of discrete constrained minimization in $V_p^l(\Tom)$] \label{lem_stab}
The discrete Poincar\'e inequalities~\eqref{c_onto_disc_abstr} are equivalent to the stability of~\eqref{eq_min} in the sense that
\begin{equation}\label{eq_stab}
    \|u\lT^*\|_{L^2(\omega)} \le \CPd{l} h_\omega \| r\lT \|_{L^2(\omega)} \qquad \forall l\in\{1{:}2\}.
\end{equation}
\end{lemma}

\begin{proof}
The Euler optimality conditions for~\eqref{eq_min} allow for the following equivalent rewriting of~\eqref{eq_min}:
\begin{equation}\label{eq_min_EL} \left \{ \begin{aligned}
&\text{Find } u\lT^* \in V_p^l(\Tom) \text{ with } d^l u\lT^* = r\lT \text{ such that }\\
&\bl u\lT^*, v\lT\br_\omega = 0 \quad \forall v\lT \in V_p^l(\Tom) \text{ with } d^l v\lT = 0.
\end{aligned} \right.
\end{equation}
Thus, \eqref{c_onto_disc_abstr} readily implies~\eqref{eq_stab}. Conversely, if~\eqref{eq_stab} holds, given any $u\lT \in V_p^l(\Tom)$ satisfying the orthogonality constraints in~\eqref{c_onto_disc_abstr}, one considers the constrained minimization problem~\eqref{eq_min} with the datum $r\lT \eq d^l u\lT$. Since $u\lT^*=u\lT$ by uniqueness, \eqref{eq_stab} implies~\eqref{c_onto_disc_abstr}.
\end{proof}

In the two cases $l\in\{1{:}2\}$, the minimization~\eqref{eq_min} writes, for $\bbr\lT \in \curl (\Ne_p(\Tom) \cap \Hcrl)$ and $r\lT \in \pol_p(\Tom)$, respectively, as  
\begin{equation} \label{eq_H_curl_div_min}
    \bu\lT^* = \Argmin_{\substack{\bv\lT \in \Ne_p(\Tom) \cap \Hcrl \\ \curl \, \bv\lT = \bbr\lT}} \|\bv\lT\|_{\bL^2(\omega)}^2 
    \, \text{ and } \,
    \bu\lT^* = \Argmin_{\substack{\bv\lT \in \Rt_p(\Tom) \cap \Hdiv \\ \dive \, \bv\lT = r\lT}} \|\bv\lT\|_{\bL^2(\omega)}^2.
\end{equation}
Lemma~\ref{lem_stab} then states that the discrete Poincar\'e inequalities~\eqref{c_onto_I_disc}
are equivalent to the stabilities
\begin{equation}\label{eq_H_curl_div_stab}
    \|\bu\lT^*\|_{\bL^2(\omega)} \le \CPd{1} h_\omega \| \bbr\lT \|_{\bL^2(\omega)} 
    \, \text{ and } \,
    \|\bu\lT^*\|_{\bL^2(\omega)} \le \CPd{2} h_\omega \| r\lT \|_{L^2(\omega)}.
\end{equation}

\subsection{Equivalence with discrete inf-sup conditions} \label{sec:equiv_inf_sup}

Let $l\in\{1{:}2\}$, $r\lT \in d^l (V_p^l(\Tom))$ and consider the following problem:
\begin{equation}\label{eq_min_EL_Lag} \left \{ \begin{alignedat}{4}
&\text{Find } u\lT^* \in V_p^l(\Tom) \text{ and } s\lT^* & & \in d^l (V_p^l(\Tom)) \text{ such} && \text{ that }\\
&\bl u\lT^*, v\lT\br_\omega - \bl s\lT^*, d^l v\lT\br_\omega & & = 0 &&\forall v\lT \in V_p^l(\Tom), \\
& \bl d^l u\lT^*, t\lT\br_\omega & & = \bl r\lT, t\lT\br_\omega &&\forall t\lT \in d^l (V_p^l(\Tom)),
\end{alignedat} \right.\end{equation}
which is called a mixed formulation of~\eqref{eq_min_EL}. As the differential operator $d^l$ is surjective from $V_p^l(\Tom)$ onto $d^l (V_p^l(\Tom))$ by definition, the Euler conditions~\eqref{eq_min_EL} are equivalent to the mixed formulation~\eqref{eq_min_EL_Lag}.
We now recall that the discrete Poincar\'e inequalities~\eqref{c_onto_disc_abstr} are equivalent to the {\em discrete inf-sup conditions}
\begin{equation}\label{eq_inf_sup}
\inf_{t\lT \in d^l (V_p^l(\Tom))} \,\, \sup_{v\lT \in V_p^l(\Tom)} \frac{\bl t\lT, d^l v\lT\br_\omega}{\|t\lT\|_{L^2(\omega)}\|v\lT\|_{L^2(\omega)}} \geq \frac{1}{\CPd{l}h_\omega} \qquad \forall l\in\{1{:}2\}.
\end{equation}

\begin{lemma}[Equivalence of~\eqref{c_onto_disc_abstr} with the discrete inf-sup conditions]\label{lem_IS} 
The discrete Poincar\'e inequalities~\eqref{c_onto_disc_abstr} are equivalent to the discrete inf-sup conditions~\eqref{eq_inf_sup}.
\end{lemma}

\begin{proof} 
Let $l\in\{1{:}2\}$.
Since~\eqref{c_onto_disc_abstr} is equivalent to the stability property~\eqref{eq_stab} as per Lemma~\ref{lem_stab}, we prove the equivalence between~\eqref{eq_stab} and~\eqref{eq_inf_sup}. 

(1) Assume the stability property~\eqref{eq_stab}. 
Let $t\lT \in d^l (V_p^l(\Tom))$. Consider, as in~\eqref{eq_min_EL}, the following well-posed problem:
\begin{equation*} \left \{ \begin{aligned}
&\text{Find }v\lT \in V_p^l(\Tom) \text{ with } d^l v\lT = t\lT \text{ such that }\\
&\bl v\lT, w\lT\br_\omega = 0 \quad \forall w\lT \in V_p^l(\Tom) \text{ with } d^l w\lT = 0.
\end{aligned} \right.\end{equation*}
The stability property~\eqref{eq_stab} gives 
$\|v\lT\|_{L^2(\omega)} \le \CPd{l} h_\omega \| t\lT \|_{L^2(\omega)}.$
Now, since $d^l v\lT = t\lT$, we infer from this bound that
\[
    \bl t\lT, d^l v\lT\br_\omega = \|t\lT\|_{L^2(\omega)}^2 \geq \frac{\|v\lT\|_{L^2(\omega)} \|t\lT\|_{L^2(\omega)}}{\CPd{l} h_\omega},
\]
which gives the discrete inf-sup condition~\eqref{eq_inf_sup}. 

(2) Conversely, we now suppose~\eqref{eq_inf_sup} and show that this implies~\eqref{eq_stab}. Let $r\lT \in d^l (V_p^l(\Tom))$ and let $u\lT^*$ solve~\eqref{eq_min}. Since~\eqref{eq_min} is equivalent to~\eqref{eq_min_EL} which is in turn equivalent to~\eqref{eq_min_EL_Lag}, we can consider $s\lT^*\in d^l (V_p^l(\Tom))$ so that the pair $(u\lT^*,s\lT^*)$ solves~\eqref{eq_min_EL_Lag}. Using in~\eqref{eq_min_EL_Lag} the test functions $v\lT = u\lT^*$ and $t\lT = s\lT^*$ and summing the two equations, we infer that
\[
    \|u\lT^*\|_{L^2(\omega)}^2 = \bl r\lT, s\lT^*\br_\omega \leq \|r\lT\|_{L^2(\omega)} \|s\lT^*\|_{L^2(\omega)},
\]
where we used the Cauchy--Schwarz inequality. Now, the discrete inf-sup condition~\eqref{eq_inf_sup} gives the existence of $v\lT \in V_p^l(\Tom)$ such that
\[
    \|s^*\lT\|_{L^2(\omega)} \leq \CPd{l}h_\omega \frac{\bl s\lT^*, d^l v\lT\br_\omega}{\|v\lT\|_{L^2(\omega)}}.
\]
From the first equation in~\eqref{eq_min_EL_Lag} and the Cauchy--Schwarz inequality, we obtain
\[
    \frac{\bl s\lT^*, d^l v\lT\br_\omega}{\|v\lT\|_{L^2(\omega)}} = \frac{\bl u\lT^*, v\lT\br_\omega}{\|v\lT\|_{L^2(\omega)}} \leq \|u\lT^*\|_{L^2(\omega)}.
\]
Combining the three above inequalities, \eqref{eq_stab} follows.
\end{proof}

In the two cases $l\in\{1{:}2\}$, the discrete inf-sup conditions~\eqref{eq_inf_sup} respectively write as
\begin{equation}\label{eq_H_curl_inf_sup}
\inf_{\bt\lT \in \curl (\Ne_p(\Tom) \cap \Hcrl)} \,\, \sup_{\bv\lT \in \Ne_p(\Tom) \cap \Hcrl} \frac{\bl \bt\lT, \curl \, \bv\lT\br_\omega}{\|\bt\lT\|_{\bL^2(\omega)}\|\bv\lT\|_{\bL^2(\omega)}} \geq \frac{1}{\CPd{1}h_\omega}
\end{equation}
and
\begin{equation}\label{eq_H_div_inf_sup}
\inf_{t\lT \in \pol_{p}(\Tom)} \,\, \sup_{\bv\lT \in \Rt_p(\Tom) \cap \Hdiv} \frac{\bl t\lT, \dive \, \bv\lT\br_\omega}{\|t\lT\|_{L^2(\omega)}\|\bv\lT\|_{\bL^2(\omega)}} \geq \frac{1}{\CPd{2}h_\omega}.
\end{equation}
By Lemma~\ref{lem_IS}, they are equivalent to the discrete Poincar\'e inequalities~\eqref{c_onto_I_disc}.

\begin{remark}[Norms]
We stress that we do not use here the norms for which the spaces are Hilbert spaces, but merely $L^2(\omega)$- or $\bL^2(\omega)$-norms, in contrast to the usual practice, see, \eg, \cite[Theorem~4.2.3]{BBF13} or~\cite[Theorem~49.13]{EG_volII}, but similarly to, \eg, \cite[Theorem~5.9]{Voh_un_apr_apost_MFE_10} or~\cite[Remark 51.12]{EG_volII}. 
\end{remark}

\subsection{Equivalence with bounds on operator norms of piecewise polynomial vector potential operators}

Let $l\in\{1{:}2\}$ and $r\lT \in d^l (V_p^l(\Tom))$. 
%Recall the notation $d^l$ from~\eqref{eq_dl}. 
A piecewise polynomial vector potential is any field $\Phi^l\lT(r\lT) \in V_p^l(\Tom)$ such that $d^l (\Phi^l\lT(r\lT)) = r\lT$, and we say that
\begin{equation}
    \Phi^l\lT: d^l (V_p^l(\Tom)) \rightarrow V_p^l(\Tom)
\end{equation}
is a {\em piecewise polynomial vector potential operator} (piecewise polynomial right-inverse of the $d^l$ operator, also called Poincar\'e map). We are particularly interested in the $L^2(\omega)$-norm minimizing operators
\begin{equation} \label{eq_pot_op}
  \Phi^{l,*}\lT(r\lT) \eq \Argmin_{\substack{v\lT \in V_p^l(\Tom) \\ d^l v\lT = r\lT}} \|v\lT\|_{L^2(\omega)}^2,
\end{equation}
with operator norm
\begin{equation} \label{eq_op_norm}
    \bigtnorm{\Phi^{l,*}\lT} \eq \max_{v\lT \in V_p^l(\Tom)} \frac{\|\Phi^{l,*}\lT(d^l v\lT)\|_{L^2(\omega)}}{\|d^l v\lT\|_{L^2(\omega)}}.
\end{equation}

\begin{lemma}[Equivalence of the best constant in~\eqref{c_onto_disc_abstr} with the operator norm of the minimal discrete vector potential operator]\label{lem_op_norm}
The operator norm $\bigtnorm{\Phi^{l,*}\lT}$ from~\eqref{eq_op_norm} {\em equals} the best discrete Poincar\'e inequality constant $\CPd{l} h_\omega$ from~\eqref{c_onto_disc_abstr}.
\end{lemma}

\begin{proof}
Observe that~\eqref{eq_pot_op} matches exactly the form of the constrained minimization~\eqref{eq_min} and use Lemma~\ref{lem_stab}. 
\end{proof}

In the $\Ne_p(\Tom) \cap \Hcrl$ setting, the $\bL^2(\omega)$-norm minimizing potential operator is
\begin{equation} \label{eq_H_curl_pot_op}
  \Phi^{\crl,*}\lT(\bbr\lT) \eq \Argmin_{\substack{\bv\lT \in \Ne_p(\Tom) \cap \Hcrl \\ \curl \, \bv\lT = \bbr\lT}} \|\bv\lT\|_{\bL^2(\omega)}^2,
\end{equation}
and its operator norm is
\begin{equation} \label{eq_op_norm_Hc}
    \bigtnorm{\Phi^{\crl,*}\lT} \eq \max_{\bv\lT \in \Ne_p(\Tom) \cap \Hcrl} \frac{\|\Phi^{\crl,*}\lT(\curl \, \bv\lT)\|_{\bL^2(\omega)}}{\|\curl \, \bv\lT\|_{\bL^2(\omega)}}.
\end{equation}
In the $\Rt_p(\Tom) \cap \Hdiv$ setting, the $\bL^2(\omega)$-norm minimizing potential operator is
\begin{equation} \label{eq_H_div_pot_op}
    \Phi^{\dive,*}\lT(r\lT) \eq \Argmin_{\substack{\bv\lT \in \Rt_p(\Tom) \cap \Hdiv \\ \dive \, \bv\lT = r\lT}} \|\bv\lT\|_{\bL^2(\omega)}^2,
\end{equation}
with operator norm 
\begin{equation} \label{eq_op_norm_Hdv}
\bigtnorm{\Phi^{\dive,*}\lT} \eq \max_{\bv\lT \in \Rt_p(\Tom) \cap \Hdiv} \frac{\|\Phi^{\dive,*}\lT(\dive \, \bv\lT)\|_{\bL^2(\omega)}}{\|\dive \, \bv\lT\|_{L^2(\omega)}}.
\end{equation}
By Lemma~\ref{lem_op_norm}, these operator norms are equivalent to the best constants $\CPd{l} h_\omega$ in the discrete Poincar\'e inequalities~\eqref{c_onto_I_disc}.

\section{Continuous Poincar\'e inequalities} \label{sec_Poinc_cont}

In this section, we state the continuous Poincar\'e inequalities and give some 
pointers to the literature for bounds on the continuous Poincar\'e constants. This will pave the way to our main topic, the discrete Poincar\'e inequalities.

\begin{proposition}[Continuous Poincar\'e inequalities] \label{prop_Poinc}
\textup{(i)} Continuous Poincar\'e inequalities without boundary conditions: There exist constants $\CP^l$, $l\in\{0{:}2\}$, only depending on the shape of $\omega$, such that
\begin{alignat}{2}
\|u\|_{L^2(\omega)} &\le \CP^l h_\omega \| d^l u \|_{L^2(\omega)} \qquad && \forall u \in V^l(\omega) \text{ such that } \bl u , v\br_\omega = 0 \nn \\
& && \qquad \forall v \in V^l(\omega) \text{ with } d^l v = 0 \qquad \forall l\in\{0{:}2\}. \label{c_onto}
\end{alignat}
\textup{(ii)} Continuous Poincar\'e inequalities with boundary conditions: There exist constants $\mCP^l$, $l\in\{0{:}2\}$, only depending on the shape of $\omega$, such that
\begin{alignat}{2}
\|u\|_{L^2(\omega)} &\le \mCP^l h_\omega \| d^l u \|_{L^2(\omega)} \qquad && \forall u \in \mV^l(\omega) \text{ such that } \bl u , v\br_\omega = 0 \nn \\
& && \qquad \forall v \in \mV^l(\omega) \text{ with } d^l v = 0 \qquad \forall l\in\{0{:}2\}. \label{c_ontom}
\end{alignat}
\end{proposition}

\begin{remark}[Proofs without explicit bounds on constants]
\textup{(i)} One known route from the literature to establish the inequalities~\eqref{c_onto}--\eqref{c_ontom} is to invoke a \emph{compactness} argument, which can be formalized in the following Peetre--Tartar lemma~\cite[Lemma~A.20]{ErnGuermondbook}: Let $X,Y,Z$ be three Banach spaces, let $A\in\calL(X;Y)$ be an injective operator, and let $T\in\calL(X;Z)$ be a compact operator. Assume that there is $\gamma>0$ such that $\gamma\|u\|_X \le \|A(u)\|_Y+\|T(u)\|_Z$ for all $u\in X$. Then there is $\alpha>0$ such that
\begin{equation} \label{eq:Peetre_Tartar}
\alpha \|u\|_X \le \|A(u)\|_Y \qquad \forall u\in X.
\end{equation}
The Peetre--Tartar lemma can be combined with a (simple and natural) 
scaling argument in the definition of the norms
to make the constant $\alpha$ in~\eqref{eq:Peetre_Tartar} nondimensional.
To briefly illustrate, let us prove~\eqref{c_onto} for $l=0$. 
We set $X\eq\{ u \in \Hgrd \;|\; \bl u, 1\br_\omega = 0\}$, $\bY\eq\bL^2(\omega)$,
$Z \eq L^2(\omega)$, $\bA(u) \eq h_\omega \grad \, u$, and $T(u):=u$. The operator $\bA$ is injective since any $u\in X$ such that $\bA(u)=\bzero$ is $L^2$-orthogonal to itself and thus vanishes identically. Moreover, $T$ is compact since the embedding $H^1(\omega) \hookrightarrow L^2(\omega)$ is compact. Finally, we have $\|u\|_X^2 = \|u\|_{\Hgrd}^2 = \|u\|_{L^2(\omega)}^2 + h_\omega^2 \|\grad \, u\|_{\bL^2(\omega)}^2 = \|T(u)\|_Z^2 + \|\bA(u)\|_{\bY}^2$. Thus, by the Peetre--Tartar Lemma, \eqref{c_onto} for $l=0$ holds true. The proof for the other Poincar\'e inequalities is similar. In particular, for the curl and divergence operators, one invokes the compactness of the embeddings 
$\Hcrl \cap \Hdivz \hookrightarrow \bL^2(\omega)$ and $\Hcrlz \cap \Hdiv \hookrightarrow \bL^2(\omega)$, see~\cite[Theorem~2]{Costabel_1990}, 
\cite[Theorem~3.1]{Birman_Solomyak_1987}, \cite[Proposition~3.7]{AmBDG:98} and~\cite{Weber1980}.
\textup{(ii)} Another, somewhat related, route to prove the Poincar\'e inequalities hinges 
on Helmholtz decompositions which show that the following operators are isomorphisms
(see, e.g., \cite[Lemma~2.8 \& Remark~2.11]{Ern_Guermond_SINUM_2023} and the references
therein):
\begin{subequations}\begin{align}
\grad: {}&\{u \in \Hgrd \;|\; \bl u , 1\br_\omega = 0\} \longrightarrow \nn \\
& \hspace*{3cm} \{\bw \in \bL^2(\omega) \;|\; \bl \bw , \bv\br_\omega = 0 
\quad \forall \bv \in \Hdivz \,\text{s.t.}\, \div \bv=0\}, \\
\curl: {}&\{\bu \in \Hcrl \;|\; \bl \bu , \bv\br_\omega = 0 
\quad \forall \bv \in \Hcrl \,\text{s.t.}\, \curl \, \bv=\bzero\} \longrightarrow\nn \\
& \hspace*{3cm} \{\bw \in \bL^2(\omega) \;|\; \bl \bw , \bv\br_\omega = 0 
\quad \forall \bv \in \Hcrlz \,\text{s.t.}\, \curl \, \bv=\bzero\}, \\
\dive: {}&\{\bu \in \Hdiv \;|\; \bl \bu , \bv\br_\omega = 0 
\quad \forall \bv \in \Hdiv \,\text{s.t.}\, \dive \, \bv=0\} \longrightarrow \{w \in L^2(\omega)\},
\end{align}\end{subequations}
with similar isomorphisms in the case of prescribed boundary conditions. 
Then, the range
of all these operators is closed, and Banach's Closed Range theorem 
(see, e.g., \cite[Lemma~C.39]{EG_volII}) implies the Poincar\'e inequalities~\eqref{c_onto}--\eqref{c_ontom}. 
\textup{(iii)} If $\omega$ is star-shaped with respect to a ball, upper bounds on the continuous Poincar\'e constants $\CP^l$, $\mCP^l$, $l\in\{0{:}2\}$, can be derived from estimates on suitable right inverses (Bogovskii/Poincar\'e integral operators) of the adjoint differential operator (see, e.g., \cite{Duran_Bog_const_12}). A generalization  of the results in~\cite{Duran_Bog_const_12} to other differential operators can be found in~\cite{Guz_Salg_Bog_const_21}, see also~\cite{Chau_Licht_Voh_Poinc_de_Rham_25}.
\end{remark}

\begin{remark}[Proofs with explicit bounds on constants] \label{rem_litr}
\textup{(i)} Inequalities~\eqref{c_onto} and~\eqref{c_ontom} for $l=0$ are the well-known 
Poincar\'e inequalities. They can be shown constructively, as, e.g., in~\cite{Pay_Wei_Poin_conv_60, Beben_Poin_conv_03} or~\cite[Exercise 22.3]{ErnGuermondbook}, from where it follows that $\CP^0 = 1 / \pi$ if $\omega$ is convex  and $\mCP^0 \leq 1$. For general nonconvex domains with a finite convex cover, upper bounds on $\CP^0$ can be found in, \eg\, \cite[Lemma~3.7]{Eym_Gal_Her_00}.
\textup{(ii)} Computable upper bounds on the continuous Poincar\'e constants $\CP^l$, $\mCP^l$, $l\in\{1{:}2\}$ can be derived by considering a suitably enumerated (shellable) shape-regular mesh $\Tom$ of $\omega$ and determining these bounds in terms of the shape-regularity parameter $\rho_{\Tom}$ and the number of elements $|\Tom|$. This approach is detailed in~\cite{Chau_Licht_Voh_Poinc_de_Rham_25}, see also the references therein.
\end{remark}

\begin{remark}[Comparison of the continuous Poincar\'e constants] 
One has $\CP^2 = \mCP^0$, $\mCP^2 = \CP^0$, and $\CP^1 = \mCP^1$. 
% We recall these identities in Remark~\ref{rem_CP2} below
We refer the reader to~\cite{PauVa:20} and the references therein for further insight into the relations between, and values of, the constants
in~\eqref{c_onto} and~\eqref{c_ontom}, including the case where boundary conditions are enforced only on part of the boundary of $\omega$. 
\end{remark}

\section{Discrete Poincar\'e inequalities} \label{sec_Poinc_disc}

In this section, we present our main result on the discrete 
Poincar\'e inequalities. We focus on the dependency of the discrete
Poincar\'e constants on the continuous-level constants $\CP^l$, $\mCP^l$, $l\in\{0{:}2\}$, and the shape-regularity parameter $\rho_{\Tom}$, the number of 
tetrahedra $|\Tom|$, and the polynomial degree $p$. 
We shall consider in Section~\ref{sec_dis_Poinc_proofs} three routes
to prove these inequalities, each one leading to different dependencies of the constants.

\subsection{Triangulations and finite element stars}
\label{sec:triang_stars}

Some specific tetrahedral meshes $\Tom$ will be of particular interest. 
We either look at $\Tom$ as a triangulation of some computational domain $\omega$, or we consider $\Tom$ as some local (vertex, edge, face) star of a shape-regular simplicial mesh $\Th$ of some larger three-dimensional computational domain $\Omega$ (open, bounded, connected, Lipschitz polyhedral set).
Let $\tau$ be a tetrahedron from $\Tom$. We will call a ``twice-extended element star'' a collection of such tetrahedra $\tau'$ from $\Tom$ which either share a vertex with $\tau$, $\tau \cap \tau' \neq \emptyset$, or such that there exists a tetrahedron $\tau''$ from $\Tom$ such that $\tau'$ shares a vertex with $\tau''$ and $\tau''$ shares a vertex with $\tau$.
As a specific case, we will consider triangulations $\Tom$ where all the domains of twice-extended element stars are Lipschitz and with a contractible closure. 

\subsection{Main result}

Our main result is as follows.

\begin{theorem}[Discrete Poincar\'e inequalities] \label{thm_disc_Poinc}
\textup{(i)} Discrete Poincar\'e inequalities without boundary conditions: There exist constants $\CPd{l}$, $l\in\{0{:}2\}$, such that
\begin{alignat}{2}
\|u\lT\|_{L^2(\omega)} &\le \CPd{l} h_\omega \| d^l u\lT \|_{L^2(\omega)} \qquad && \forall u\lT \in V_p^l(\Tom) \text{ such that } \bl u\lT , v\lT\br_\omega = 0 \nn \\
& && \qquad \forall v\lT \in V_p^l(\Tom) \text{ with } d^l v\lT = 0 \qquad \forall l\in\{0{:}2\}. \label{onto}
\end{alignat}
\textup{(ii)} Discrete Poincar\'e inequalities with boundary conditions: 
There exist constants $\mCPd{l}$, $l\in\{0{:}2\}$, such that
\begin{alignat}{2}
\|u\lT\|_{L^2(\omega)} &\le \mCPd{l} h_\omega \| d^l u\lT \|_{L^2(\omega)} \qquad && \forall u\lT \in \mV_p^l(\Tom) \text{ such that } \bl u\lT , v\lT\br_\omega = 0 \nn \\
& && \qquad \forall v\lT \in \mV_p^l(\Tom) \text{ with } d^l v\lT = 0 \qquad \forall l\in\{0{:}2\}. \label{ontom}
\end{alignat}
Here, the constants $\CPd{l}$, $\mCPd{l}$ have the following properties:
\begin{enumerate} 
\item \label{it_1} $\CPd{0}\le \CP^0$ and $\mCPd{0}\le \mCP^0$. Thus, $\CPd{0} \le 1 / \pi$ if $\omega$ is convex, and $\mCPd{0} \leq 1$ for any $\omega$, see the discussion in Remark~\ref{rem_litr}.

\item \label{it_2} If $\overline \omega$ is contractible, then there exist constants $C_{\mathrm{min}}^l$, $l\in\{1{:}2\}$, only depending on the shape-regularity parameter $\rho_{\Tom}$ and the number of tetrahedra $|\Tom|$, such that $\CPd{l} \le C_{\mathrm{min}}^l \CP^{l}$. If $\Tom$ is a vertex or edge star, then there exist constants $C_{\mathrm{min}}^l$, $l\in\{1{:}2\}$, only depending on the shape-regularity parameter $\rho_{\Tom}$, such that $\CPd{l} \le C_{\mathrm{min}}^l \CP^{l}$ and $\mCPd{l} \le C_{\mathrm{min}}^l \mCP^{l}$.
    
\item \label{it_3} There exist constants $C_{\mathrm{st}}^l$, $l\in\{1{:}2\}$, only depending on the shape-regularity parameter $\rho_{\Tom}$ and the polynomial degree $p$, such that $\CPd{l} \le C_{\mathrm{st}}^l \CP^{l}$ and $\mCPd{l} \le C_{\mathrm{st}}^l \mCP^{l}$. Moreover, if all domains of twice-extended element stars in $\Tom$ are Lipschitz and with a contractible closure, then $C_{\mathrm{st}}^2$ only depends on the shape-regularity parameter $\rho_{\Tom}$.

\item \label{it_4} The constants $\CPd{l}$, $\mCPd{l}$, $l\in\{1{:}2\}$, admit upper bounds that only depend on the shape-regularity parameter $\rho_{\Tom}$, the number of tetrahedra $|\Tom|$, and the polynomial degree $p$, but that do not need to invoke the constants $\CP^{l}$, $\mCP^{l}$.

\end{enumerate}

\end{theorem}

\subsection{Discussion} Let us discuss items~\eqref{it_2} to~\eqref{it_4} of Theorem~\ref{thm_disc_Poinc}:

\begin{itemize}
  
\item {\bf Discussion of~\eqref{it_2}.} This result is established in Section~\ref{sec_route_1} below and relies on piecewise polynomial extension operators. 
    The constants here are systematically $p$-robust, but can unfavorably depend on the number of tetrahedra in $\Tom$. In stars or extended stars or any local patches, $|\Tom|$ is bounded as a function of the shape-regularity parameter $\rho_{\Tom}$, leading to discrete Poincar\'e constants only depending on $\rho_{\Tom}$ and the continuous Poincar\'e constants $\CP^{l}$ or $\mCP^{l}$, $l\in\{1{:}2\}$ (for which upper bounds only depending on the shape-regularity parameter $\rho_{\Tom}$ can be derived as discussed in Remark~\ref{rem_litr}). The assumption that $\overline \omega$ is contractible is automatically satisfied if $\Tom$ is a vertex or edge star. For a general domain $\Om$ with a mesh $\Th$, local stars with contractible $\omega$ are supposed in, \eg, \cite{Falk_Winth_loc_coch_14, Arn_Guz_loc_stab_L2_com_proj_21}. This assumption does not request the whole computational domain $\Om$ to be contractible, but merely the local star-domains $\omega$. For example, for a domain $\Om$ with a hole, there may be local stars with non-contractible corresponding $\overline \omega$ if $\Th$ is rather coarse, but typically all local star-domains $\omega$ are of contractible closure on finer meshes. We refer for further details to the recent discussion in~\cite[Remark~2.1]{EGPV_HO:25}.
  
\item {\bf Discussion of~\eqref{it_3}.} This result is proven in Section~\ref{sec_route_2} below upon relying on stable commuting projections. This is probably the most common way of proving the discrete Poincar\'e inequalities. In this case, the constants $\CPd{l}$, $\mCPd{l}$, $l\in\{1{:}2\}$, are independent of the number of tetrahedra in $\Tom$ (i.e., this number can be arbitrarily high), but may (unfavorably) depend on the polynomial degree $p$.
    In the $\Hdiv$ setting ($l=2$), the $p$-robust projector from~\cite[Definition~3.3]{Demk_Voh_loc_glob_H_div_25} gives a constant $\CPd{2}$ independent of both the number of tetrahedra in $\Tom$ and the polynomial degree $p$ if the domains of all twice-extended element stars in $\Tom$ are Lipschitz and with a contractible closure (as discussed above, this is typically satisfied in practice, at least for sufficient mesh refinement). To our knowledge, this is the best result available so far. Once again, upper bounds on the continuous Poincar\'e constants only depending on the shape-regularity parameter $\rho_{\Tom}$ can be derived as discussed in Remark~\ref{rem_litr}. 

\item {\bf Discussion of~\eqref{it_4}.} This result is established in Section~\ref{sec_route_3} below. The technique of proof does not rely on the continuous Poincar\'e inequalities. Therefore, the upper bounds on the discrete Poincar\'e constants do not involve here the constants $\CP^l$, $\mCP^l$, $l\in\{1{:}2\}$. There are no requirements on the triangulation $\Tom$ either (as $\overline\omega$ being contractible or $\omega$ being a local star). The direct proof argument leads to discrete Poincar\'e constants depending on the shape-regularity parameter $\rho_{\Tom}$, the number of tetrahedra $|\Tom|$, and the polynomial degree $p$.
\end{itemize}

\begin{remark}[Extension to any space dimension in the framework of finite element exterior calculus] 
The results on continuous Poincar\'e inequalities of Proposition~\ref{prop_Poinc} extend to any space dimension in the framework of finite element exterior calculus, see~\cite{arnold2006finite, arnold2010finite, arnold2018finite,Guz_Salg_Bog_const_21,Chau_Licht_Voh_Poinc_de_Rham_25} and the references therein.
The same holds true for the results on discrete Poincar\'e inequalities of Theorem~\ref{thm_disc_Poinc}, points~\eqref{it_1} and~\eqref{it_4}, where the proofs do not use any specific information on the space dimension and the operator $d^l$. The results of points~\eqref{it_2} and~\eqref{it_3}, instead, are obtained by invoking specific results for the $\curl$ and $\dive$ operators and are currently only available in three space dimensions as stated in Theorem~\ref{thm_disc_Poinc} (as well as in two space dimensions up to straightforward adaptations). 
\end{remark}

\section{Proofs of discrete Poincar\'e inequalities}\label{sec_dis_Poinc_proofs}

In this section, we describe the three routes mentioned above to prove the discrete Poincar\'e inequalities~\eqref{onto} and~\eqref{ontom} for $l\in\{1{:}2\}$, leading to Theorem~\ref{thm_disc_Poinc}. Recall that these three routes respectively consist in:
\begin{enumerate}
\item Invoking equivalence between discrete and continuous minimizers (piecewise polynomial extension operators);
\item Invoking stable commuting projections with stability in $L^2$ for data whose image by $d^l$ is piecewise polynomial (we also comment on stability in graph spaces and fractional-order Sobolev spaces);
\item Invoking piecewise Piola transformations.
\end{enumerate}
We observe that the first two routes hinge on the continuous Poincar\'e inequalities~\eqref{c_onto} and~\eqref{c_ontom} for $l\in\{1{:}2\}$, whereas the third route employs only a finite-dimensional argument. The different routes give the different dependencies of the discrete Poincar\'e constant on the parameters $\rho_{\Tom}$, $|\Tom|$, and $p$, as summarized in items~\eqref{it_2}--\eqref{it_4} of Theorem~\ref{thm_disc_Poinc}.
For routes 1 and 2, we give pointers to the literature providing tools to realize the proofs, whereas we present a stand-alone proof for route~3.

\subsection{Unified presentation}

We introduce some more (unified) notation, building on Section~\ref{sec_comp_not}. Let $l\in\{0{:}2\}$.We define the kernels of the differential operators
\begin{align} \label{eq:spaces_dz} 
\Zz V^l(\omega) & \eq \{u \in V^l(\omega) \;:\; d^l u=0\},
\end{align}%
where we notice that $\Zz V^0(\omega)=\{u\in V^0(\omega) \;:\; u = \text{constant}\}$. We also define their $L^2$-orthogonal complements
\begin{align} \label{eq:spaces_dz_perp} 
\Zz^\perp V^l(\omega) &\eq \{u\in V^l(\omega)\;:\; \bl u , v\br_\omega = 0 \quad \forall v \in \Zz V^l(\omega)\}, 
\end{align}%
where we notice that $\Zz^\perp V^0(\omega)=\{u\in V^0(\omega) \;:\; \bl u,1\br_\omega=0\}$.
We define the spaces $\Zz \mV^l(\omega)$ as
in~\eqref{eq:spaces_dz} (notice that $\Hgrdzgrdz = \{0\}$), and 
their $L^2$-orthogonal complements $\Zz^\perp \mV^l(\omega)$ as in~\eqref{eq:spaces_dz_perp}
(notice that $\Zz^\perp \mV^0(\omega)=\mV^0(\omega)$).

We define similarly the kernels of the differential operators in the discrete spaces~\eqref{eq:local_spaces}. We namely set
\begin{align} \label{eq:def_ZZ} 
\Zz V_p^l(\Tom) &\eq  \{ u\lT \in  V_p^l(\Tom): d^l u\lT=0 \}, 
\end{align}%
whereas the $L^2$-orthogonal complements are defined as
\begin{align} \label{eq:def_ZZ_perp} 
\Zz^{\perp} V_p^l(\Tom) &\eq \{ u\lT \in  V_p^l(\Tom): \bl u\lT,v\lT \br_{\omega} =0 \quad \forall v\lT \in  \Zz V_p^l(\Tom)\}. 
\end{align}
We define the subspaces $\Zz \mV_p^l(\Tom)$ as well as $\Zz^\perp \mV_p^l(\Tom)$ similarly. 

Finally, to unify the notation regarding boundary conditions, we set, for all $l\in\{1{:}2\}$,
\begin{subequations}\begin{alignat}{4}
{\widetilde V}^l(\omega)& \eq V^l(\omega) \; & &\text{or} \; \mV^l(\omega),&\qquad
{\widetilde V}_p^l(\Tom)& \eq V_p^l(\Tom) \; & & \text{or} \; \mV_p^l(\Tom),\\
\Zz^{\perp} \widetilde V^l(\omega)&\eq\Zz^\perp V^l(\omega) \; & & \text{or} \; \Zz^\perp \mV^l(\omega),&\qquad
{\Zz^{\perp} \widetilde V}_p^l(\Tom)&\eq\Zz^{\perp} V_p^l(\Tom) \; & & \text{or} \; \Zz^{\perp} \mV_p^l(\Tom).
\end{alignat}\end{subequations}
Then, the continuous Poincar\'e inequalities~\eqref{c_onto}--\eqref{c_ontom} are rewritten as follows:
\begin{equation} \label{c_onto_abstract} 
\|u\|_{L^2(\omega)} \le \tCP^l h_\omega \| d^l u \|_{L^2(\omega)}  \qquad \forall u \in {\Zz^{\perp} \widetilde V}^l(\omega)  \qquad \forall l\in\{0{:}2\},
\end{equation}
with $\tCP^l\eq\CP^l$ or $\mCP^l$ depending on the context, and the discrete  
Poincar\'e inequalities~\eqref{onto}--\eqref{ontom} are rewritten as follows:
\begin{equation} \label{onto_abstract} 
\|u\lT\|_{L^2(\omega)} \le \tCPd{l} h_\omega \| d^l u\lT \|_{L^2(\omega)} \qquad \forall u\lT \in {\Zz^{\perp} \widetilde V}_p^l(\Tom) \qquad \forall l\in\{0{:}2\},
\end{equation}
with $\tCPd{l}\eq\CPd{l}$ or $\mCPd{l}$ depending on the context.

\subsection{Route 1: Invoking equivalence between discrete and continuous minimizers} \label{sec_route_1}

Let $l\in\{1{:}2\}$.
For all $r\lT \in d^l({\widetilde V}_p^l(\Tom)) \subset {\widetilde V}_p^{l+1}(\Tom)$, as in~\eqref{eq_min}, we consider the following two 
constrained quadratic minimization problems:
\begin{subequations}\label{eq:argmin_dm}\begin{align} 
u\lT^* & \eq \Argmin_{\substack{v\lT \in {\widetilde V}_p^l(\Tom) \\ d^l v\lT = r\lT}} \|v\lT\|_{L^2(\omega)}^2, \label{eq:argmin_dm_disc}\\
u^* & \eq \Argmin_{\substack{v \in {\widetilde V}^l(\omega) \\ d^l v = r\lT}} \|v\|_{L^2(\omega)}^2.\label{eq:argmin_dm_cont}
\end{align}\end{subequations}
We notice that the finite-dimensional 
minimization set $\{v\lT \in {\widetilde V}_p^l(\Tom) \;:\; d^l v\lT = r\lT\}$ 
is nonempty, closed, and convex, and so is also the larger, infinite-dimensional set 
$\{v \in {\widetilde V}^l(\omega) \;:\; d^l v = r\lT\}$. Thus, as for~\eqref{eq_min}, both problems admit a unique minimizer. Moreover, we trivially have
\[
\|u^*\|_{L^2(\omega)} \le \|u^*\lT\|_{L^2(\omega)}.
\]
The Euler optimality conditions respectively read, \cf\ \eqref{eq_min_EL}: 
\begin{equation}\label{eq:min_EL_disc} \left\{ \begin{aligned} 
&\text{Find } u\lT^* \in {\widetilde V}_p^l(\Tom) \text{ with } d^l u\lT^* = r\lT \text{ such that }\\
&\bl u\lT^*, v\lT\br_\omega = 0 \quad \forall v\lT \in {\widetilde V}_p^l(\Tom) \text{ with } d^l v\lT = 0
\end{aligned}\right. \end{equation}
and
\begin{equation}\label{eq:min_EL} \left\{ \begin{aligned}
&\text{Find } u^* \in {\widetilde V}^l(\omega) \text{ with } d^l u^* = r\lT \text{ such that }\\
&\bl u^*, v\br_\omega = 0 \quad \forall v \in {\widetilde V}^l(\omega) \text{ with } d^l v = 0.
\end{aligned}\right. \end{equation}

\begin{lemma}[Discrete Poincar\'e inequalities invoking equivalence between discrete and continuous minimizers]
Let $l\in\{1{:}2\}$.
Assume that there is $C_{\mathrm{min}}^l$ such that, for all $r\lT \in d^l({\widetilde V}_p^l(\Tom))$, the solutions to~\eqref{eq:argmin_dm} satisfy
\begin{equation} \label{ass:min}
\|u^*\lT\|_{L^2(\omega)} \le C_{\mathrm{min}}^l \|u^*\|_{L^2(\omega)}.
\end{equation}
Then~\eqref{onto_abstract} holds true with constant $\tCPd{l}\le C_{\mathrm{min}}^l \tCP^l$.
\end{lemma}

\begin{proof}
Let $u\lT\in {\Zz^{\perp} \widetilde V}_p^l(\Tom)$. Set $r\lT \eq d^lu\lT$. Since ${\Zz^{\perp} \widetilde V}_p^l(\Tom)\subset {\widetilde V}_p^l(\Tom)$, we have 
$r\lT\in d^l({\widetilde V}_p^l(\Tom))$. Moreover, by considering the Euler conditions~\eqref{eq:min_EL_disc} and~\eqref{eq:min_EL}, we infer that $u^*\lT \in {\Zz^{\perp} \widetilde V}_p^l(\Tom)$ and 
$u^* \in {\Zz^{\perp} \widetilde V}^l(\omega)$. In addition, since the minimization problems admit a unique solution, and since $u\lT$ satisfies the Euler conditions for the discrete problem, we have $u\lT=u^*\lT$. Invoking~\eqref{ass:min} followed by the continuous Poincar\'e inequality~\eqref{c_onto_abstract} gives
\[
\|u\lT\|_{L^2(\omega)} = \|u^*\lT\|_{L^2(\omega)} \le C_{\mathrm{min}}^l \|u^*\|_{L^2(\omega)}
\le C_{\mathrm{min}}^l \tCP^l h_\omega \|d^l u^*\|_{L^2(\omega)}.
\]
Since, from~\eqref{eq:argmin_dm}, $\|d^l u^*\|_{L^2(\omega)}=\|r\lT\|_{L^2(\omega)}=\|d^l u\lT\|_{L^2(\omega)}$, we conclude that~\eqref{onto_abstract} holds true with constant $\tCPd{l}\le C_{\mathrm{min}}^l \tCP^l$.
\end{proof}

In the case of homogeneous boundary conditions, the minimizations~\eqref{eq:argmin_dm} take the form
\begin{equation} \label{eq_min_curl_div}
    u\lT^* \eq \Argmin_{\substack{v\lT \in \mV_p^l(\Tom) \\ d^l v\lT = r\lT}} \|v\lT\|_{L^2(\omega)}^2, 
    \qquad
    u^* \eq \Argmin_{\substack{v \in \mV^l(\omega) \\ d^l v = r\lT}} \|v\|_{L^2(\omega)}^2
\end{equation}
with data $r\lT \in d^l (\mV_p^l(\Tom))$. In the case $l=2$ (divergence operator), \eqref{ass:min} has been established in~\cite[Corollaries~3.3 and~3.8]{Ern_Voh_p_rob_3D_20} whenever $\Tom$ is a vertex star, see also~\cite[Proposition~3.1 and Corollary~4.1]{Chaum_Voh_p_rob_3D_H_curl_24}. 
In the case $l=1$ (curl operator), \eqref{ass:min} has been established in~\cite[Proposition~6.6]{Chaum_Ern_Voh_Maxw_22}
for edge stars and in~\cite[Theorem~3.3 and Corollary~4.3]{Chaum_Voh_p_rob_3D_H_curl_24} for vertex stars.

In the case without boundary conditions, the minimizations~\eqref{eq:argmin_dm} take the form
\begin{equation} \label{eq_min_curl_div_ext}
    u\lT^* \eq \Argmin_{\substack{v\lT \in V_p^l(\Tom) \\ d^l v\lT = r\lT}} \|v\lT\|_{L^2(\omega)}^2, 
    \qquad
    u^* \eq \Argmin_{\substack{v \in V^l(\omega) \\ d^l v = r\lT}} \|v\|_{L^2(\omega)}^2
\end{equation}
with data $r\lT \in d^l (V_p^l(\Tom))$.
When the mesh $\Tom$ is more complex than a vertex star, \cite[Theorem~C.1]{Demk_Voh_loc_glob_H_div_25} gives the desired result~\eqref{ass:min} in the case $l=2$ (divergence operator) under the assumption that $\overline \omega$ is contractible. The case $l=1$ (curl operator) can be treated similarly.

One interesting outcome of the proofs based on route 1 is that
the constant $C_{\mathrm{min}}^l$, and consequently $\tCPd{l}$, is independent of the polynomial degree $p$. Still, $C_{\mathrm{min}}^l$ and $\tCPd{l}$ depend on $|\Tom|$ and the shape-regularity parameter $\rho_{\Tom}$, leading to item~\eqref{it_2} of Theorem~\ref{thm_disc_Poinc} (recall that for stars or any local patch, $|\Tom|$ only depends on $\rho_{\Tom}$).

\subsection{Route 2: Invoking stable commuting projections} \label{sec_route_2}

Here, we proceed as in~\cite{Fort_MFEs_77}, \cite[Theorem~5.11]{arnold2006finite},
\cite[Theorem~3.6]{arnold2010finite}, \cite{Christ_Wint_sm_proj_08}, \cite[Proposition~5.4.2]{BBF13}, \cite{Falk_Winth_loc_coch_14}, and~\cite[Theorem~44.6 \& Remark~51.12]{EG_volII}.

\begin{lemma}[Discrete Poincar\'e inequalities invoking stable commuting projections] \label{lem:route2}
Assume that there are projections
$\Pi_p^m:{\widetilde V}^m(\omega) \rightarrow {\widetilde V}_p^m(\Tom)$, $m\in\{1{:}3\}$, satisfying, for all $l\in\{1{:}2\}$, the commuting
property
\begin{equation} \label{ass:commut} 
d^l(\Pi_p^l(u)) = \Pi_p^{l+1}(d^l u) \qquad \forall u\in {\widetilde V}^l(\omega),
\end{equation}
and the $L^2$-stability property on data whose image by $d^l$ is piecewise polynomial
\begin{equation} \label{eq:bound_L2} 
\|\Pi_p^l(u)\|_{L^2(\omega)} \leq C_{\mathrm{st}} \|u\|_{L^2(\omega)} \qquad \forall u\in {\widetilde V}^l(\omega) \text{ such that } d^l u \in {\widetilde V}_p^{l+1}(\Tom).
\end{equation}
Then~\eqref{onto_abstract} holds true with constant $\tCPd{l}\le C_{\mathrm{st}} \tCP^l$.
\end{lemma}

\begin{proof}
Let $u\lT\in {\Zz^{\perp} \widetilde V}_p^l(\Tom)$. Set $r\lT \eq d^l u\lT$. We consider again the
minimization problems in~\eqref{eq:argmin_dm}. Recall that both problems
are well-posed and that $u\lT=u^*\lT$. We observe that $\Pi_p^l(u^*)\in {\widetilde V}_p^l(\Tom)$ 
by definition, that $d^l u^* = r\lT \in {\widetilde V}_p^{l+1}(\Tom)$, and that
\[
d^l(\Pi_p^l(u^*))=\Pi_p^{l+1}(d^l u^*)=\Pi_p^{l+1}(r\lT)=r\lT,
\]
where we used the commuting property~\eqref{ass:commut}, the fact that 
$r\lT \in {\widetilde V}_p^{l+1}(\Tom)$, and that $\Pi_p^{l+1}$ is a projection. 
This shows that $\Pi_p^l(u^*)$ is in the discrete minimization set. Using the $L^2$-stability property~\eqref{eq:bound_L2} and the continuous Poincar\'e inequality~\eqref{c_onto_abstract}, we infer that
\begin{equation}\label{eq_stab_com}\begin{split}
\|u\lT\|_{L^2(\omega)} = \|u^*\lT\|_{L^2(\omega)} &\le \|\Pi_p^l(u^*)\|_{L^2(\omega)} \\
&\le C_{\mathrm{st}} \|u^*\|_{L^2(\omega)} \\
&\le C_{\mathrm{st}} \tCP^l h_\omega \|d^l u^*\|_{L^2(\omega)}.
\end{split}\end{equation}
Since $\|d^l u^*\|_{L^2(\omega)}=\|r\lT\|_{L^2(\omega)}=\|d^l u\lT\|_{L^2(\omega)}$, we conclude that~\eqref{onto_abstract} holds true with constant $\tCPd{l}\le C_{\mathrm{st}} \tCP^l$.
\end{proof}

Operators satisfying~\eqref{ass:commut}--\eqref{eq:bound_L2} have been constructed in~\cite[Definition~3.1]{Ern_Gud_Sme_Voh_loc_glob_div_22} (for $l=2$) and in~\cite[Definition~2]{Chaum_Voh_H_curl_proj_24} for $l=1$. 
In all these cases, $C_{\mathrm{st}}$ depends on the shape-regularity parameter $\rho_{\Tom}$ and on the polynomial degree $p$, but is independent of the number of tetrahedra in $\Tom$. 
In the $\Hdiv$ setting ($l=2$), the $p$-robust projector of~\cite[Definition~3.3]{Demk_Voh_loc_glob_H_div_25} gives a constant $C_{\mathrm{st}}$ only depending on the shape-regularity parameter $\rho_{\Tom}$ if the domains of all twice-extended element stars in $\Tom$ are Lipschitz and with a contractible closure. All these cases are summarized in item~\eqref{it_3} of Theorem~\ref{thm_disc_Poinc}.

\begin{remark}[Example of a stable $\Hdiv$ commuting projection]
Let us rewrite the $\Hdiv$ construction from~\cite[Definition~3.1]{Ern_Gud_Sme_Voh_loc_glob_div_22} in the present setting so as to give an idea on how the key properties~\eqref{ass:commut}--\eqref{eq:bound_L2} follow. Let thus $l=2$, and consider for instance the case without boundary conditions. Let $\bu \in \bV^2(\omega) = \Hdiv$ be given. The construction of $\Pi_p^2 (\bu)$ proceeds in three steps.

\begin{enumerate}
\item \label{it_equil_1} On each tetrahedron $\tau \in \Tom$, one considers the $\bL^2(\tau)$-orthogonal projection of $\bu|_\tau$ onto the Raviart--Thomas space $\Rt_p(\tau)$ (see~\eqref{eq_RT}) under a divergence constraint,
\begin{equation}\label{eq_tauh}
    \bxih|_\tau \eq \Argmin_{\substack{\bv\lT \in \Rt_p(\tau)\\ \dive \, \bv\lT = \Pi_p^3(\dive \, \bu)}}
    \|\bu - \bv\lT\|_{\bL^2(\tau)},
\end{equation} 
where $\Pi_p^3$ is the $L^2(\omega)$-orthogonal projection onto $V_p^3(\Tom) = \pol_{p}(\Tom)$ (note that this projection is elementwise, since $\pol_{p}(\Tom)$ is a space of piecewise polynomials without any continuity requirement across the mesh interfaces). As there is no normal trace prescription for $\bxih$, it belongs to the discontinuous piecewise polynomial space $\Rt_p(\Tom)$ only.

\item \label{it_equil_2} For each vertex $\ver$ from the vertex set of $\Tom$, $\ver \in \Vom$, let $\Tha$ be the vertex star (all tetrahedra of $\Tom$ sharing $\ver$). 
One defines the Raviart--Thomas polynomial $\frh^\ver \in \Rt_p(\Tha) \cap \Hdva$ such that
\begin{equation}\label{eq_sigma_a}
\frh^\ver
\eq
\Argmin_{\substack{
\bv\lT \in \Rt_p(\Tha) \cap \Hdva
\\
\dive \, \bv\lT= \Pi_p^3(\psia \dive \, \bu + \grad \, \psia {\cdot} \bxih)
}}
\|\RTproj{p} (\psia \bxih) - \bv\lT\|_{\bL^2(\oma)}.
\end{equation} 
Here, $\psia$ is the hat basis function, the piecewise affine scalar-valued function taking value $1$ at the vertex $\ver$ and $0$ at the other vertices of $\Tom$, $\Hdva$ is the subspace of $\bH(\dive,\oma)$ with homogenous normal component over the faces where $\psia$ vanishes, and $\RTproj{p}$ is the elementwise canonical Raviart--Thomas projector (applied to the piecewise polynomial $\psia \bxih$, so that its action is well defined).

\item \label{it_equil_3} Finally, one extends $\frh^\ver$ by zero outside of the patch subdomain $\oma$ and defines $\Pi_p^2 (\bu) \in \bV_p^2(\Tom) = \Rt_p(\Tom) \cap \Hdiv$ via
\begin{equation}\label{eq_sigma}
    \Pi_p^2 (\bu) \eq \sum_{\ver \in \Vom} \frh^\ver.
\end{equation}
\end{enumerate}

Step~\eqref{it_equil_1} above projects $\bu$ onto the finite-dimensional space $\Rt_p(\Tom)$. 
Steps~\eqref{it_equil_2} and~\eqref{it_equil_3} above amount to the so-called flux equilibration from a posteriori error analysis~\cite{Dest_Met_expl_err_CFE_99,Brae_Pill_Sch_p_rob_09,Ern_Voh_p_rob_15}. Owing to the partition of unity by the hat functions,
\begin{equation}\label{eq_PU}
  \sum_{\ver \in \Vom} \psia = 1,
\end{equation}
the commuting property~\eqref{ass:commut} is straightforward since
\[
    \dive \, \Pi_p^2 (\bu) \refv{\eqref{eq_sigma}}{=} \sum_{\ver \in \Vom} \dive \,  \frh^\ver
    \refv{\eqref{eq_sigma_a}}{=} \sum_{\ver \in \Vom} \Pi_p^3(\psia \dive \, \bu + \grad \, \psia {\cdot} \bxih) \refv{\eqref{eq_PU}}{=} \Pi_p^3(\dive \, \bu).
\]
The projection property amounts to $\Pi_p^2 (\bu) = \bu$ if $\bu \in \Rt_p(\Tom) \cap \Hdiv$. This follows easily from the following three arguments: 1) $\bxih = \bu$ in~\eqref{eq_tauh}; 2) $\frh^\ver = \RTproj{p} (\psia \bu)$ in~\eqref{eq_sigma_a}, since the elementwise canonical Raviart--Thomas projector gives $\RTproj{p} (\psia \bu) \in \Rt_p(\Tha) \cap \Hdva$ and its commuting property implies that $\dive \, \RTproj{p} (\psia \bu) = \Pi_p^3(\dive \, (\psia \bu)) = \Pi_p^3(\psia \dive \, \bu + \grad \, \psia {\cdot} \bu)$; 3) we conclude that
$\Pi_p^2 (\bu) \eq \sum_{\ver \in \Vom} \frh^\ver = \sum_{\ver \in \Vom} \RTproj{p} (\psia \bu) = \bu$
from~\eqref{eq_sigma} and~\eqref{eq_PU}. 
Finally, the stability property~\eqref{eq:bound_L2} is proven in~\cite[Theorem~3.2, property~(3.7)]{Ern_Gud_Sme_Voh_loc_glob_div_22} using the stability of the vertex star problems~\eqref{eq_sigma_a} and the obvious stability of the elementwise problems~\eqref{eq_tauh}.
\end{remark}

% \pra{Note that in~\cite{Arn_Guz_loc_stab_L2_com_proj_21}, a projection satisfying~\eqref{ass:commut} and~\eqref{eq:bound_L2} was constructed, however, it relies on discrete Poincar\'e inequalities on extended stars (\cite[Section 2.2]{Arn_Guz_loc_stab_L2_com_proj_21})}. 

\begin{remark}[$L^2$-stability of $\Pi_p^l$]
The assumptions in Lemma~\ref{lem:route2} on the projection $\Pi_p^l$ do not ask for
full stability in $L^2(\omega)$. Indeed, it suffices that $\Pi_p^l$ be defined on
the graph space ${\widetilde V}^l(\omega)$ and that the $L^2$-stability property~\eqref{eq:bound_L2}
holds true for functions so that $d^l u \in {\widetilde V}_p^{l+1}(\Tom)$ 
(and $d^l u$ is, in particular, a polynomial). 
\end{remark}

\begin{remark}[Graph-norm stability of $\Pi_p^l$]\label{rem_graph_stab}
Actually, the proof still works if 
one considers commuting projections that are stable in the graph norm
\begin{equation}\label{eq_graph_norm}
    \|v\|_{\widetilde V^l(\omega)}\eq\big( \|v\|_{L^2(\omega)}^2 + h_\omega^2 \|d^l v\|_{L^2(\omega)}^2 \big)^{\frac12},
\end{equation}
leading to the bound $\tCPd{l}\le C_{\mathrm{st}}\big(1+(\tCP^l)^2\big)^{\frac12}$.
Indeed, the final step~\eqref{eq_stab_com} of the above proofs now writes
\begin{align*}
\|u\lT\|_{L^2(\omega)} = \|u^*\lT\|_{L^2(\omega)} &\le \|\Pi_p^l(u^*)\|_{L^2(\omega)} \\
&\le C_{\mathrm{st}} \|u^*\|_{\widetilde V^l(\omega)} \\
&\le C_{\mathrm{st}} \big(1+(\tCP^l)^2\big)^{\frac12} h_\omega \|d^l u\lT\|_{L^2(\omega)}.
\end{align*}
\end{remark}

\begin{remark}[Fractional-order Sobolev stability of $\Pi_p^l$]
It is also possible to invoke regularity results stating that
$\Zz^\perp \widetilde V^l(\omega) \hookrightarrow H^s(\omega)$, $s>\frac12$,
with embedding constant $C_{\mathrm{emb}}$ so that 
\[
    \|v\|_{H^s(\omega)} \le C_{\mathrm{emb}} h_\omega \|d^l v\|_{L^2(\omega)} \qquad \forall 
v\in \Zz^\perp \widetilde V^l(\omega),
\]
where
\[
    \|v\|_{H^s(\omega)}^2 = \|v\|_{L^2(\omega)}^2 + h_\omega^s|v|_{H^s(\omega)}^2,\qquad |v|_{H^s(\omega)}^2 = \int_\omega\int_\omega \frac{|v(x)-v(y)|^2}{|x-y|^{3+2s}}dxdy.
\]
(Again, the scaling by $h_\omega$ is introduced for dimensional consistency.)
This allows one to consider commuting projections that are stable only in 
$H^s(\omega)$, $s>\frac12$, i.e.,
\[
    \|\Pi_p^l(z)\|_{L^2(\omega)} \leq C_{\mathrm{st}} \|z\|_{H^s(\omega)} \qquad \forall z\in  
    H^s(\omega).
\] 
The proof of the discrete Poincar\'e inequality then runs as follows.
For all $u\lT\in \Zz^{\perp} {\widetilde V}_p^l(\Tom)$, there exists 
$z\in \Zz^\perp \widetilde V^l(\omega)$ such that $d^l z=d^l u\lT$ (indeed,
take $z \eq u\lT-m$, where $m$ is the $L^2$-orthogonal projection of $u\lT$
onto $\Zz\widetilde V^l(\omega)$). We have
\[
\|u\lT\|_{L^2(\omega)}^2 = \bl u\lT,\Pi^l_p(z)\br_\omega + 
\bl u\lT,u\lT-\Pi^l_p(z)\br_\omega = \bl u\lT,\Pi^l_p(z)\br_\omega,
\]
since $u\lT-\Pi^l_p(z)\in \Zz \widetilde V_p^l(\Tom)$ (indeed,
$d^l(u\lT-\Pi^l_p(z))=d^l u\lT-\Pi_p^{l+1}(d^l z)=d^l u\lT-\Pi_p^{l+1}(d^l u\lT)=0$ since $\Pi_p^{l+1}$ leaves
$\widetilde V_p^{l+1}(\Tom)$ pointwise invariant).
The above identity together with the Cauchy--Schwarz inequality gives
\[
    \|u\lT\|_{L^2(\omega)} \le \|\Pi_p^l(z)\|_{L^2(\omega)}.
\]
Observing that $z\in H^s(\omega)$, we infer that
\begin{align*}
\|u\lT\|_{L^2(\omega)} &\le \|\Pi_p^l(z)\|_{L^2(\omega)} \le 
C_{\mathrm{st}} \|z\|_{H^s(\omega)} \le C_{\mathrm{st}} C_{\mathrm{emb}} h_\omega\|d^l z\|_{L^2(\omega)}
= C_{\mathrm{st}} C_{\mathrm{emb}} h_\omega \|d^l u\lT\|_{L^2(\omega)},
\end{align*}
which proves~\eqref{onto_abstract} with constant $\tCPd{l}\le C_{\mathrm{st}} C_{\mathrm{emb}}$. The above approach was considered in early works where $L^2$-stable or graph-stable commuting projections were not yet available. The idea is to trade some stability of $\Pi^p_l$ by invoking subtle regularity results on the curl and divergence operators. On the downside, estimating $\tCPd{l}$ now requires upper bounds on $C_{\mathrm{st}}$ and $C_{\mathrm{emb}}$. The present remark may be of interest for the sake of an historical perspective. The interested reader can find more details in~\cite{Boff_Cost_Daug_Demk_Hipt_inf_sup_Maxw_p_11} and the references therein.
\end{remark}

With the above developments, we can now add one more equivalent statement for discrete Poincar\'e inequalities, in the spirit of~\cite[Theorems~3.6 and~3.7]{arnold2010finite}. This completes the results on equivalent statements given in Section~\ref{sec_equivs}.

\begin{lemma}[Equivalence of discrete Poincar\'e inequalities with the existence of graph-stable commuting projections] \label{lem_eq_stab_proj_dP}
The discrete Poincar\'e inequalities~\eqref{onto_abstract} for $l\in\{1{:}2\}$ are equivalent to the existence of projections $\Pi_p^m:{\widetilde V}^m(\omega) \rightarrow {\widetilde V}_p^m(\Tom)$, $m\in\{1{:}3\}$, satisfying, for all $l\in\{1{:}2\}$, the commuting property
\begin{equation} \label{ass:commut_gen} 
d^l(\Pi_p^l(u)) = \Pi_p^{l+1}(d^l u) \qquad \forall u\in {\widetilde V}^l(\omega)
\end{equation}
and the graph-stability property
\begin{equation} \label{eq:bound_graph} 
\|\Pi_p^l(u)\|_{\widetilde V^l(\omega)} \leq C_{\mathrm{st}} \|u\|_{\widetilde V^l(\omega)} \qquad \forall u\in {\widetilde V}^l(\omega).
\end{equation}
\end{lemma}

\begin{proof} We show the two implications.

\textup{(i)} That the existence of graph-stable projections satisfying~\eqref{ass:commut_gen}--\eqref{eq:bound_graph} implies the discrete Poincar\'e inequalities~\eqref{onto_abstract} for $l\in\{1{:}2\}$ follows from Lemma~\ref{lem:route2} and Remark~\ref{rem_graph_stab}.

\textup{(ii)} Suppose the validity of the discrete Poincar\'e inequalities~\eqref{onto_abstract}. We show that this implies the existence of projections satisfying~\eqref{ass:commut_gen}--\eqref{eq:bound_graph}.
A generic way is to take $\Pi_p^3$ as the $L^2$-orthogonal projection onto $\widetilde V_p^3(\Tom)$ and to define $\Pi_p^l:\widetilde V^l(\omega)\rightarrow \widetilde V_p^l(\Tom)$ for all $l\in\{1{:}2\}$ by the following constrained quadratic minimization problems, similar to~\eqref{eq_min} and~\eqref{eq:argmin_dm_disc}: For all $u\in \widetilde V^l(\omega)$,
\begin{equation} \label{eq:def_Pi_min}
\Pi_p^l(u) \eq \Argmin_{\substack{v\lT \in \widetilde V_p^l(\Tom) \\ d^l v\lT = \Pi_p^{l+1}(d^l u)}} \|u-v\lT\|_{L^2(\omega)}^2,
\end{equation}
first for $l=2$ and then for $l=1$. Notice that the commuting property~\eqref{ass:commut_gen} is built in the definition of $\Pi^p_l$, so that only the stability in the graph norm~\eqref{eq:bound_graph} needs to be verified. To this purpose, we notice that the Euler optimality conditions for~\eqref{eq:def_Pi_min}, as in~\eqref{eq_min_EL} and~\eqref{eq:min_EL_disc}, read as follows: Find $\Pi_p^l(u)\in \widetilde V_p^l(\Tom)$ with $d^l(\Pi_p^l(u)) = \Pi_p^{l+1}(d^l u)$ such that
\[
\bl \Pi_p^l(u) - u,v\lT \br_\omega=0 \qquad \forall v\lT\in \widetilde V_p^l(\Tom) \text{ with } d^l v\lT=0.
\]
The mixed formulation using a Lagrange multiplier, as in~\eqref{eq_min_EL_Lag}, reads as follows: Find $\Pi_p^l(u)\in \widetilde V_p^l(\Tom)$ and $s\lT \in d^l(\widetilde V_p^l(\Tom))$ such that
\begin{alignat*}{4}
&\bl\Pi_p^l(u),v\lT\br_\omega - \bl s\lT,d^l v\lT\br_\omega &&= \bl u,v\lT\br_\omega&\qquad &\forall v\lT\in \widetilde V_p^l(\Tom), \\
&\bl d^l(\Pi_p^l(u)),t\lT\br_\omega &&= \bl d^l u,t\lT\br_\omega&\qquad &\forall t\lT\in d^l(\widetilde V_p^l(\Tom)).
\end{alignat*}
(Notice that $\bl \Pi_p^{l+1}(d^l u),t\lT\br_\omega=\bl d^l u,t\lT\br_\omega$ owing to the Euler optimality conditions for $\Pi_p^{l+1}$ and the fact that $d^{l+1}t\lT=0$). As highlighted in Section~\ref{sec:equiv_inf_sup}, the discrete Poincar\'e inequality~\eqref{onto_abstract} is equivalent to to the discrete inf-sup condition formulated using $L^2$-norms, see Lemma~\ref{lem_IS}. The inf-sup condition in the form of~\eqref{eq_inf_sup} readily implies the discrete inf-sup condition in the graph norm
\[
\inf_{t\lT \in d^l(\widetilde V_p^l(\Tom))} \,\, \sup_{v\lT \in {\widetilde V}_p^l(\Tom)} \frac{\bl t\lT,d^l v\lT\br_\omega}{\|t\lT\|_{L^2(\omega)}\|v\lT\|_{\widetilde V^l(\omega)}} \geq \frac{1}{\big(1+(\CPd{l})^2\big)^{\frac12} h_\omega}.
\]
Then, invoking~\cite[Theorem~4.2.3]{BBF13} or~\cite[Theorem~49.13]{EG_volII}, we obtain
\begin{align*}
    \|\Pi_p^l(u)\|_{\widetilde V^l(\omega)} & \leq \|u\|_{\widetilde V^l(\omega)} + 2 \big(1+(\CPd{l})^2\big)^{\frac12} h_\omega \|d^l u\|_{L^2(\omega)}\\
    & \leq \big(10+8(\CPd{l})^2\big)^{\frac12} \|u\|_{\widetilde V^l(\omega)}.
\end{align*}
This proves that the commuting projection $\Pi_p^l$ defined above is indeed stable in the graph norm. 
\end{proof}

\begin{remark}[Locality] The above graph-stable commuting projections are not necessarily locally defined and locally stable. Stable {\em local} commuting projections are designed in~\cite{Falk_Winth_loc_coch_14, Arn_Guz_loc_stab_L2_com_proj_21, Ern_Gud_Sme_Voh_loc_glob_div_22, Chaum_Voh_H_curl_proj_24, Demk_Voh_loc_glob_H_div_25}, see also the references therein. 
\end{remark}

\subsection{Route 3: Invoking piecewise Piola transformations} \label{sec_route_3}

In this section, we prove the discrete Poincar\'e inequality by a direct argument, thereby circumventing the need to invoke the continuous Poincar\'e inequalities. The discrete Poincar\'e constants resulting from the present proofs depend on the shape-regularity parameter $\rho_{\Tom}$, the number of tetrahedra $|\Tom|$, and the polynomial degree $p$, as summarized in item~\eqref{it_4} of Theorem~\ref{thm_disc_Poinc}. The proof shares ideas with the one given in~\cite{EGPV_HO:25}, but eventually employs a different argument to conclude.

The starting point, shared with~\cite{EGPV_HO:25}, is to introduce
reference meshes and piecewise Piola transformations on those meshes.
We enumerate the set of vertices (resp., edges, faces, and cells (tetrahedra)) 
in $\Tom$ as $\Vom\eq\{ v_1, \ldots, v_{N^{\mathrm{v}}}\}$
(resp., $\Eom\eq\{e_1,\ldots,e_{N^{\mathrm{e}}}\}$, 
$\Fom\eq\{f_1,\ldots,f_{N^{\mathrm{f}}}\}$, and $\Tom\eq\{\tau_1,\ldots,\tau_{N^{\mathrm{c}}}\}$ 
with $N^{\mathrm{c}}=|\Tom|$. All these geometric objects are oriented by
increasing vertex enumeration (see, e.g., \cite[Chapter~10]{ErnGuermondbook}). 
The topology and orientation of the mesh $\Tom$ is completely described by the connectivity arrays 
\begin{subequations} \label{eq_connect} \begin{align}
&\jev:\{1{:}N^{\mathrm{e}}\}\times\{0{:}1\}\rightarrow \{1{:}N^{\mathrm{v}}\}, \\
&\jfv:\{1{:}N^{\mathrm{f}}\}\times\{0{:}2\}\rightarrow \{1{:}N^{\mathrm{v}}\}, \\ 
&\jcv:\{1{:}N^{\mathrm{c}}\}\times\{0{:}3\}\rightarrow \{1{:}N^{\mathrm{v}}\},
\end{align} \end{subequations}
such that $\jev(m,n)$ is the global vertex number of the vertex $n$ of the edge $e_m$,
and so on (the local enumeration of vertices is by increasing enumeration order). 
Notice that the connectivity arrays only take integer values and are independent of the
actual coordinates of the vertices in the physical space $\R^3$. 

Let $\rho_\sharp>0$ be a positive real number and let $T_\sharp$ be a (finite) integer number.
The number of meshes with shape-regularity parameter bounded from above by $\rho_\sharp$ and cardinality bounded from above by $T_\sharp$ with different possible realizations of the connectivity arrays is bounded from above by a constant $\hat N_{\sharp}\eq\hat N(\rho_\sharp,T_\sharp)$ only depending on $\rho_\sharp$ and $T_\sharp$.
Thus, for each $\rho_\sharp$ and $T_\sharp$, there is a \emph{finite} set of 
reference meshes, which we denote by $\wTT\eq \wTT(\rho_\sharp,T_\sharp)$,
such that every mesh $\mathcal{T}$ with the shape-regularity parameter bounded from above by $\rho_\sharp$ and cardinality bounded from above by $T_\sharp$ has the same connectivity arrays 
as those of one reference mesh in the set $\wTT$. 
We enumerate the reference meshes in $\wTT$ as $\{\wT_1,\ldots,\wT_{\hat N_\sharp}\}$ and fix them once and for all. 
For each reference mesh, the element diameters are of order unity, and the shape-regularity parameter is chosen as small as possible (it is bounded from above by $\rho_\sharp$). 
For all $j\in\{1{:}\hat N_{\sharp}\}$, we let $\womega_j$ be the open, bounded, connected, 
Lipschitz polyhedral set covered by the reference mesh $\wT_j$. 
For all $l\in\{1{:}2\}$, we define the piecewise polynomial spaces $V_p^l(\wT_j)$
and $\mV_p^l(\wT_j)$ as in~\eqref{eq:local_spaces}, and
set ${\widetilde V}_p^l(\wT_j) \eq V_p^l(\wT_j)$ or $\mV_p^l(\wT_j)$ depending on whether
boundary conditions are enforced or not. 
We also define ${\Zz^{\perp} \widetilde V}_p^l(\wT_j)$ as the $L^2$-orthogonal complement of
the kernel subspace $\{\wu\lT\in {\widetilde V}_p^l(\wT_j) \;:\; d^l \wu\lT = 0\}$ in ${\widetilde V}_p^l(\wT_j)$.
Norm equivalence in finite dimension implies that, for all $j\in\{1{:}\hat N_{\sharp}\}$ and all $p\ge0$,
there exists a constant $\tCP^l(\wT_j,p)$ such that
\begin{equation} \label{onto_abstract_w}
\|\wu\lT\|_{L^2(\womega_{j})} \le \tCP^l(\wT_j,p)  \| d^l \wu\lT \|_{L^2(\womega_{j})}  \qquad \forall \wu\lT \in {\Zz^{\perp} \widetilde V}_p^l(\wT_j).
\end{equation}

Consider an arbitrary mesh $\Tom$ with shape-regularity parameter bounded from above by $\rho_\sharp$ and cardinality bounded from above by $T_\sharp$. Then there is an index $j(\Tom) \in \{1{:}\hat N_\sharp\}$ so that $\Tom$ and $\wT_{j(\Tom)}$ share the same connectivity arrays.
Therefore, $\Tom$ can be generated from $\wT_{j(\Tom)}$ by a piecewise-affine geometric mapping
$\bF_{\Tom}\eq\{\bF_\tau:\wtau\rightarrow \tau\}_{\tau\in\Tom}$, where all the geometric mappings $\bF_\tau$ are affine, invertible, with positive Jacobian, and $\bigcup_{\tau\in\Tom}\bF_{\tau}^{-1}(\tau)=\wT_{j(\Tom)}$.
For all $\tau\in\Tom$, let $\bJ_\tau$ be the Jacobian matrix of $\bF_\tau$.
We consider the Piola transformations $\psi_{\Tom}^l:L^2(\omega)\rightarrow L^2(\womega_{j(\Tom)})$, for all $l\in\{1{:}3\}$, such that $\psi_\tau^l\eq\psi_{\Tom}^l|_\tau$ is defined as follows: For all $v\in L^2(\tau)$,
\begin{subequations} \begin{align}
\psi_\tau^1(v) &\eq \bJ_\tau^{\textsc{t}} ( v \circ \bF_\tau), \\
\psi_\tau^2(v) &\eq \det(\bJ_\tau) \bJ_\tau^{-1} ( v \circ \bF_\tau), \\
\psi_\tau^3(v) &\eq \det(\bJ_\tau) ( v \circ \bF_\tau).
\end{align} \end{subequations}
The restricted Piola transformations (we keep the same notation for simplicity)
$\psi_{\Tom}^l : {\widetilde V}_p^l(\Tom) \rightarrow {\widetilde V}_p^l(\wT_{j(\Tom)})$ are isomorphisms. This follows from the fact that $\Tom$ and $\wT_{j(\Tom)}$ have the same connectivity arrays, that $\bF_{\Tom}$ maps any edge (face, tetrahedron) in $\wT_{j(\Tom)}$ to an edge (face, tetrahedron) of $\Tom$ with the same orientation, and that, for each tetrahedron $\tau\in\Tom$, $\psi_\tau^l$ is an isomorphism that preserves appropriate moments~\cite[Lemma~9.13 \& Exercise~9.4]{ErnGuermondbook}.
Moreover, the Piola transformations satisfy the following bounds:
\begin{equation} \label{eq_is1}
    \|\psi_{\Tom}^l\|_{\calL} \eq \|\psi_{\Tom}^l\|_{\calL(L^2(\omega);L^2(\womega_{j(\Tom)}))} \le C(\rho_{\sharp})(\overline{h}_{\Tom})^{l-\frac 32},
\end{equation}
where $\overline{h}_{\Tom}$ denotes the biggest diameter of a cell in $\Tom$, and they satisfy the following commuting properties: 
\begin{equation} \label{eq:commut_Piola}
d^{l}(\psi_{\Tom}^l(v)) = \psi_{\Tom}^{l+1}(d^l v) \qquad \forall v \in {\widetilde V}^l(\omega).
\end{equation}
We use the shorthand notation $\psi_{\Tom}^{-l}$ for the inverse of the Piola transformations. 
We have 
\begin{equation} \label{eq_is2}
    \|\psi_{\Tom}^{-l}\|_{\calL} \eq \|\psi_{\Tom}^{-l}\|_{\calL(L^2(\womega_{j(\Tom)});L^2(\omega))} \le C(\rho_{\sharp})(\underline{h}_{\Tom})^{-l+\frac 32},
\end{equation}
where $\underline{h}_{\Tom}$ denotes the smallest diameter of a cell in $\Tom$. The commuting property~\eqref{eq:commut_Piola} readily gives 
\begin{equation} \label{eq:commut_Piola_inv}
    \psi_{\Tom}^{-(l+1)}(d^l \wv) = d^l(\psi_{\Tom}^{-l}(\wv)) \qquad \forall \wv\in {\widetilde V}^l(\womega_{j(\Tom)}).
\end{equation}

\begin{lemma}[Discrete Poincar\'e inequalities invoking piecewise Piola transformations] 
The discrete Poincar\'e inequalities~\eqref{onto_abstract} hold true for all $l\in\{1{:}2\}$ with a constant $\tCPd{l}$ only depending on the shape-regularity parameter $\rho_{\Tom}$, the number of tetrahedra $|\Tom|$, and the polynomial degree $p$.
\end{lemma}

\begin{proof}
Let $u\lT\in {\Zz^{\perp} \widetilde V}_{p}^l(\Tom)$ and set $r\lT \eq d^l u\lT$. As in~\eqref{eq_min} and~\eqref{eq:argmin_dm_disc}, we consider the following two (well-posed) constrained quadratic minimization problems:
\begin{equation} \label{eq:argmin_dm_wdm}
u\lT^* \eq \Argmin_{\substack{v\lT \in {\widetilde V}_p^l(\Tom) \\ d^l v\lT = r\lT}} \|v\lT\|_{L^2(\omega)}^2, 
\qquad
\wu\lT^* \eq \Argmin_{\substack{\wv\lT \in {\widetilde V}_p^l(\wT_{j(\Tom)}) \\ d^l \wv\lT = \whr\lT}} \|\wv\lT\|_{L^2(\womega_{j(\Tom)})}^2,
\end{equation}
with 
\begin{equation} \label{eq:rT}
    \whr\lT\eq \psi_{\Tom}^{l+1}(r\lT).
\end{equation}
The Euler conditions for the second minimization problem imply that $\wu\lT^* \in {\Zz^{\perp} \widetilde V}_p^l(\wT_{j(\Tom)})$. Owing to the discrete Poincar\'e inequality~\eqref{onto_abstract_w}, we infer that
\[
\|\wu\lT^*\|_{L^2(\womega_{j(\Tom)})} \le \tCP^l(\wT_{j(\Tom)},p)  \| \whr\lT \|_{L^2(\womega_{j(\Tom)})}.
\]
Moreover, we observe that $\psi_{\Tom}^{-l}(\wu\lT^*) \in {\widetilde V}_p^l(\Tom)$ and, 
owing to~\eqref{eq:commut_Piola_inv}, the constraint in the second problem in~\eqref{eq:argmin_dm_wdm}, and~\eqref{eq:rT}, we have 
\[
d^l(\psi_{\Tom}^{-l}(\wu\lT^*)) = \psi^{-(l+1)}_{\Tom}(d^l\wu\lT^*)= \psi^{-(l+1)}_{\Tom}(\whr\lT) = r\lT.
\]
Hence, $\psi_{\Tom}^{-l}(\wu\lT^*)$ is in the minimization set of the first problem in~\eqref{eq:argmin_dm_wdm}. This implies that
\begin{align*}
\|u\lT\|_{L^2(\omega)} = \|u\lT^*\|_{L^2(\omega)} &\le  \|\psi_{\Tom}^{-l}(\wu\lT^*)\|_{L^2(\omega)} \\
&\le \|\psi_{\Tom}^{-l}\|_{\calL} \|\wu\lT^*\|_{L^2(\womega_{j(\Tom)})} \\
&\le \|\psi_{\Tom}^{-l}\|_{\calL} \tCP^l(\wT_{j(\Tom)},p) \| \whr\lT \|_{L^2(\womega_{j(\Tom)})} \\
&\le \|\psi_{\Tom}^{-l}\|_{\calL}\|\psi_{\Tom}^{l+1}\|_{\calL} \tCP^l(\wT_{j(\Tom)},p) \| r\lT \|_{L^2(\omega)}\\
&= \|\psi_{\Tom}^{-l}\|_{\calL}\|\psi_{\Tom}^{l+1}\|_{\calL} \tCP^l(\wT_{j(\Tom)},p) \| d^l u\lT \|_{L^2(\omega)}.
\end{align*}
The bounds~\eqref{eq_is1}--\eqref{eq_is2} on the operator norm of the Piola maps and their inverse together give
$\|\psi_{\Tom}^{-l}\|_{\calL}\|\psi_{\Tom}^{l+1}\|_{\calL} \le C(\rho_{\Tom},|\Tom|) h_\omega$,
where we used that $\overline{h}_{\Tom} \leq h_\omega$ and $\overline{h}_{\Tom}/\underline{h}_{\Tom} \le C(\rho_{\Tom},|\Tom|)$. This implies that
\begin{align*}
\|u\lT\|_{L^2(\omega)} &\le \bigg( C(\rho_{\Tom},|\Tom|) \max_{j\in\{1{:}{\hat N}_\sharp\}} \tCP^l(\wT_j,p) \bigg) 
h_\omega \| d^l u\lT \|_{L^2(\omega)}.
\end{align*}
This completes the proof.
\end{proof}

\bibliographystyle{acm_mod}
\bibliography{references}

\end{document}